\documentclass{amsart}

\usepackage{amssymb}
 \usepackage{amsthm}

 \usepackage{lineno}

\usepackage{verbatim}
\usepackage{amsmath}
\usepackage{graphicx}
\usepackage{amssymb}
\usepackage{epstopdf}
\usepackage{color}
\usepackage{enumitem}
\usepackage{mathptmx}

\usepackage[T1]{fontenc}

\usepackage{enumitem}
\makeatletter
\newcommand{\mylabel}[2]{#2\def\@currentlabel{#2}\label{#1}}
\makeatother

\usepackage{xargs}
\usepackage[pdftex,dvipsnames]{xcolor}
\usepackage[colorinlistoftodos,prependcaption,textsize=tiny]{todonotes}
\newcommandx{\unsure}[2][1=]{\todo[linecolor=red,backgroundcolor=red!25,bordercolor=red,#1]{#2}}
\newcommandx{\change}[2][1=]{\todo[linecolor=blue,backgroundcolor=blue!25,bordercolor=blue,#1]{#2}}
\newcommandx{\info}[2][1=]{\todo[linecolor=OliveGreen,backgroundcolor=OliveGreen!25,bordercolor=OliveGreen,#1]{#2}}
\newcommandx{\improvement}[2][1=]{\todo[linecolor=Plum,backgroundcolor=Plum!25,bordercolor=Plum,#1]{#2}}
\newcommandx{\thiswillnotshow}[2][1=]{\todo[disable,#1]{#2}}

\newtheorem{theorem}{Theorem}[section]
\newtheorem{corollary}[theorem]{Corollary}
\newtheorem{lemma}[theorem]{Lemma}

\newtheorem{proposition}[theorem]{Proposition}
\newtheorem{defi}[theorem]{Definition}

\newtheorem{question}[theorem]{Question}
\newtheorem{remark}[theorem]{Remark}
\newtheorem{example}[theorem]{Example}

\newcommand{\BH}{\Bbb H}
\newcommand{\BG}{\Bbb G}
\newcommand{\BK}{\Bbb K}
\newcommand{\BL}{\Bbb L}
\newcommand{\BI}{\Bbb I}
\newcommand{\BJ}{\Bbb J}
\newcommand{\BT}{\Bbb T}

\newcommand{\bx}{\mathbf{x}}
\newcommand{\ba}{\mathbf{a}}
\newcommand{\bb}{\mathbf{b}}
\newcommand{\by}{\mathbf{y}}

\newcommand{\bw}{\mathbf{w}}

\newcommand{\arity}{k}
\newcommand{\diff}{n}
\newcommand{\seq}{t}

\newcommand{\gap}{g}

\newcommand{\Hom}{\mathrm{Hom}}
\newcommand{\diagonal}{\Delta}
\newcommand{\dist}{\mathrm{dist}}

\title[{Dismantlability, connectedness, and mixing in relational structures}]{Dismantlability, connectedness, and mixing\\in relational structures}

\author{Raimundo Brice\~no}
\address{Facultad de Matem\'aticas, Pontificia Universidad Cat\'olica de Chile, Santiago, Chile}
\email{raimundo.briceno@mat.uc.cl}
\author{Andrei Bulatov}
\address{School of Computing Science, Simon Fraser University, Burnaby BC, Canada}
\email{abulatov@sfu.ca}
\author{V\'ictor Dalmau}
\address{Department of Information and Communication Technologies, Universitat Pompeu Fabra, Barcelona, Spain}
\email{victor.dalmau@upf.edu}
\author{Beno\^it Larose}
\address{LACIM, Universit\'e du Qu\'ebec a Montr\'eal, Montr\'eal, Canada}
\email{blarose@lacim.ca}


\begin{document}

\maketitle

\begin{abstract}
The Constraint Satisfaction Problem (CSP) and its counting counterpart appears under different guises in many areas of mathematics, computer science, and elsewhere. Its structural and algorithmic properties have demonstrated to play a crucial role in many of those applications. For instance, in the decision CSPs, structural properties of the relational structures involved---like, for example, dismantlability---and their logical characterizations have been instrumental for determining the complexity and other properties of the problem. Topological properties of the solution set such as connectedness are related to the hardness of CSPs over random structures. Additionally, in approximate counting and statistical physics, where CSPs emerge in the form of spin systems, mixing properties and the uniqueness of Gibbs measures have been heavily exploited for approximating partition functions and free energy.

In spite of the great diversity of those features, there are some eerie similarities between them. These were observed and made more precise in the case of graph homomorphisms by Brightwell and Winkler, who showed that dismantlability of the target graph, connectedness of the set of homomorphisms, and good mixing properties of the corresponding spin system are all equivalent. In this paper we go a step further and demonstrate similar connections for arbitrary CSPs. This requires much deeper understanding of dismantling and the structure of the solution space in the case of relational structures, and new refined concepts of mixing introduced by Brice\~no. In addition, we develop properties related to the study of valid extensions of a given partially defined homomorphism, an approach that turns out to be novel even in the graph case. We also add  to the mix the combinatorial property of finite duality and its logic counterpart, FO-definability, studied by Larose, Loten, and Tardif.
\end{abstract}

\section{Introduction}
\label{section1}

The Constraint Satisfaction Problem (CSP) provides a powerful framework in a wide range of areas of mathematics, computer science, statistical physics, and elsewhere. The goal in a CSP is to find an assignment to variables from a given set that satisfies a number of given constraints. The counting version of the problem asks about the number of such assignments. The CSP however appears in different forms: as the standard one outlined above in AI and computer science \cite{dechter2003}, as the homomorphism problem in graph and model theory \cite{1-feder,MR2089014}, as conjunctive query evaluation in logic and database theory \cite{KolaitisVardi00-containment}, as computing the partition function of a spin system in statistical physics \cite{Mezzard09:information} and related areas, like symbolic dynamics and coding \cite{lind1995,marcus2012}.

The CSP allows for many approaches of diverse nature, and every application field exploits some of its many facets: structural properties of constraints for complexity and algorithms, probabilistic properties and the topology of the solution space in Random CSP and random structures, mixing properties in statistical physics and dynamical systems, decay of correlations and the uniqueness of probabilistic measures in approximate counting, and homomorphic duality and logical characterizations in model theory. In \cite{1-brightwell}, it was observed that some of these properties are actually closely related, at least in the simple case of graph homomorphisms. In this paper we take this research direction a step further by extending Brightwell and Winkler's results to the general CSP, and by refining and widening the range of the properties involved.

We start off with a brief introduction of the features of the CSP considered in this paper. Afterwards, we provide a detailed account of the necessary background and a description of our results. Every CSP involves a set of variables and a domain, a set of possible values for the variables. Assumptions about these two sets differ in different areas. The most studied case in combinatorics and complexity theory is when both sets are finite. However, many interesting problems such as scheduling and temporal and spatial reasoning involve infinite domains; see also the extensive literature on infinite CSPs (for example, \cite{Bodirsky15} and the references therein). In other cases such as in statistical physics, it is natural to choose the set of variables to be infinite (a lattice, for example). Then, it is also natural to study probability distributions over such assignments---where \emph{Gibbs measures} and the problem of their \emph{(non-)uniqueness} appear naturally \cite{1-georgii}---and also study quantities such as \emph{entropy} and \emph{free energy} \cite{bandyopadhyay2008,1-briceno}.

Following \cite{1-feder}, CSPs can be formulated as the problem of deciding the existence of a homomorphism from a finite relational structure $\BG$ to a target relational structure $\BH$, where $\BG$ and $\BH$ encode the variables and the values of the CSP. The complexity of this problem, especially the case when $\BH$ is a fixed finite relational structure, has received a lot of attention, culminating with the proof of the Feder-Vardi conjecture \cite{Bulatov17,Zhuk17}, which asserts that every CSP is either in $\mathrm{PTIME}$ or $\mathrm{NP}$-complete. In the present paper we focus as well on the case when $\BH$ is finite, although our main focus is not algorithmic but rather structural. In particular, we are interested in studying the space $\Hom(\BG,\BH)$ of homomorphisms from $\BG$ to $\BH$. Furthermore, following \cite{1-brightwell}, we consider homomorphisms from both finite and infinite relational structures $\BG$ (although \cite{1-brightwell} only considers graphs), a flexibility that turns out to be useful to see different aspects of homomorphism spaces $\Hom(\BG,\BH)$ that otherwise would be meaningless. We note here that the case where $\BG$ is restricted has also been studied. In particular, there is an important line of research that studies the homomorphism existence problem when $\BG$ belongs to a prescribed set of relational structures and $\BH$ is an arbitrary relational structure (e.g., see \cite{Grohe07-otherside}). The case where both $\BG$ and $\BH$ are restricted has also been investigated (e.g., see \cite{CohenCJZ19}).

There is a vast quantity of literature concerning graph homomorphisms and their properties through the lens of statistical physics \cite{borgs2006,2-dyer,2-brightwell}. In this context, it is very common to encode a {\em spin system} as a pair of relational structures $\BG$ and $\BH$, where $\BG$ contains a set of variables/particles and $\BH$ contains the set of values/spins that each particle could take, imposing \emph{hard constraints} on them, i.e., disregarding configurations of values that do not satisfy all the given constraints. In practical terms, all this reduces to the study of the maps from $\BG$ to $\BH$ that are homomorphisms, individually and as a set. In particular, many important parameters of a spin system such as free energy and entropy can be learned from studying such a set of homomorphisms.

In \cite{1-brightwell}, it was proved that many of the properties of graph homorphisms used in the above areas are equivalent to a single structural property of graphs, namely, \emph{dismantlability}. In this paper we follow a similar approach and study properties of CSPs over general relational structures that we put into basically three categories: (1) dismantlability, (2) connectedness, and (3) mixing. Furthermore, as a consequence of our results, we established a connection with a fourth notion not initially contemplated in \cite{1-brightwell}: (4) \emph{finite duality}.

\subsubsection*{Dismantlability}

A graph is said to be dismantlable if it can be reduced to a single vertex by removing vertices whose neighborhood is contained in the neighborhood of some other vertex. Such transformations are called \emph{folds}, and they can be viewed as \emph{retractions} of a very particular kind. Ultimately, this kind of action allow us to reduce $\BH$ and $\Hom(\BG,\BH)$ by replacing the appearance of certain spins with others that have equal or more freedom of allocation. Dismantlable graphs were introduced in \cite{1-nowakowski}, based on ideas already present in \cite{1-kelly} in the context of lattices, and have been intensively studied in combinatorics. Distamantlability can be generalized in a natural way to relational structures. Indeed, some variants of this notion have been used in the study of CSPs.  In particular, dismantlability has been applied in \cite{Bulatov12:enumerating} to the problem of enumerating all solutions of $\Hom(\BG,\BH)$ with polynomial delay. Also, it has played a major role in the study of CSPs definable in first-order logic \cite{DalmauKL04,MR2357493}.

\subsubsection*{Connectedness}

When $\BG$ is finite, it is often useful to convert $\Hom(\BG,\BH)$ into an auxiliary graph and explore the connectivity properties of it. The set of edges of $\Hom(\BG,\BH)$ can be defined in a variety of ways, usually the most suitable to the problem at hand. For example, it is common to say that two elements from $\Hom(\BG,\BH)$ are close (and therefore adjacent in the graph) if the Hamming distance between them is smaller than a certain threshold. The particular case when this threshold is $1$ has been intensively studied, motivated initially by the fact that the connectedness of the solution space for SAT problems over random instances is linked to the performance of standard satisfiability algorithms, such as WalkSAT or DPLL \cite{Achlioptas04:exponential,Krzakala07:Gibbs}. This has given rise to a general framework called {\em reconfiguration} \cite{Ito11} that goes way beyond homomorphisms (see also \cite{Nishimura18} for a recent survey). Work in this area encompasses both structural questions (under which conditions is $\Hom(\BG,\BH)$ connected?) and algorithmic ones (what is the complexity of, deciding, given $\BG$ and $\BH$ as input, whether $\Hom(\BG,\BH)$ is connected? Its diameter? The shortest path between two given members of $\Hom(\BG,\BH)$? Etc.). In the context of spin systems, the connectedness of $\Hom(\BG,\BH)$ is related to processes that consists on periodically updating the spin of a single or a small set of particles (e.g., irreducibility of \emph{Glauber dynamics}). We also consider an alternative way to define adjacency in $\Hom(\BG,\BH)$ via \emph{links} as in \cite{MR2357493}. This notion of adjacency is linked to the so-called \emph{finite duality property}, which is another of the main themes of our work.

\subsubsection*{Mixing}

Mixing properties have been intensively studied in statistical physics and related areas (see \cite{1-alon,ban2017tree,boyle2010,2-briceno,briceno2018factoring,ceccherini2012,pavlov2015}), and are usually applied when the set of particles in $\BG$ is very large or infinite. In this case, it can be very useful to be able to ``glue'' together partial homomorphisms, provided their domains are far from each other. There are several properties that formalize this phenomenon and it is common to establish hierarchies among them. More concretely, given a metric in $\BG$, it is natural to ask whether there exists some uniform gap such that for any two subsets $V$ and $W$ of particles sufficiently far apart (in terms of the gap), and for any pair of homomorphisms $\phi,\psi\in\Hom(\BG,\BH)$, we can find a third one, $\gamma$, such that restriction of $\gamma$ to $V$ and $W$ coincides with the restrictions of $\phi$ and $\psi$ on $V$ and $W$, respectively. On the contrary, whenever the information content of a given set (at least partially) determines the information content of another set (i.e., the possible values that the variables on it can take), no matter how far it is, such a phenomenon has been called \emph{long range action} in previous work \cite{4-brightwell}.

Similar phenomena are used in the related area of approximate computing of partition functions, where many algorithms are based on decay of correlations between values of remote elements of $\BG$, which allows for approximation of partition functions based only on local neighborhoods of variables \cite{1-dyer,1-weitz}.

\subsubsection*{Finite duality and logic characterizations}

Homomorphism duality often helps to design a solution algorithm for a CSP or establish useful properties. A relational structure $\BH$ is said to have \emph{homomorphism duality} if there is a set $\mathcal{O}$ of relational structures---called \emph{obstructions}---such that a relational structure $\BG$ has a homomorphism to $\BH$ if and only if no relational structure from $\mathcal{O}$ is homomorphic to $\BG$. Sometimes the set of obstructions is very simple, say, any bipartite graph has homomorphic duality, where $\mathcal{O}$ is the set of all odd cycles. If $\mathcal{O}$ can be chosen finite, we say that $\BH$ has \emph{finite duality}. 

Homomorphism duality is closely related to another property of CSPs. Let $\mathcal L$ be a logic language such as first order, second order, etc. The problem of deciding homomorphisms to a relational structure $\BH$ is said to be {\em expressible in \emph{$\mathcal L$}} if there is a formula $\Phi$ in the language $\mathcal L$ such that $\BG$ has a homomorphism to $\BH$ if and only if $\Phi$ is true on $\BG$. It is known, for instance, that $\BH$ has a set of obstructions consisting of relational structures of bounded treewidth if and only if the corresponding homomorphism problem is expressible in Datalog \cite{1-feder}, or that $\BH$ has finite duality if and only if the corresponding problem is expressible in first order logic \cite{Atserias08}. For a survey on dualities for CSP, see \cite{BulatovKL08}.

\subsection*{Our results}

In this paper we extend the approach taken in \cite{1-brightwell} to the case of homomorphisms between relational structures, that includes, besides graphs, very natural special cases, like \emph{digraphs}, \emph{hypergraphs}, and \emph{shifts of finite type} (see Example \ref{exmp:sft}). In particular, we show (Theorem \ref{the:mainsimple}) that, for a relational structure $\BH$, the three following conditions are equivalent: (A) $\BH^2$ dismantles to a substructure of its diagonal, that is, the substructure of $\BH^2$ induced by the set $\{(a,a)\mid a\in\BH\}$; (B) for any $\BG$, the homomorphism graph $\Hom(\BG,\BH)$ is connected; and (C) for any $\BG$, the space $\Hom(\BG,\BH)$ satisfies certain mixing properties.
Furthermore, there are several contexts where it is natural to work by forcing a certain subset of variables to take each of them a particular value and work with the remaining ones. This idea inspired a refined version of Theorem \ref{the:mainsimple}, namely Theorem \ref{the:main}, which can be regarded as the study of \emph{boundary long range actions}, i.e., long range action phenomena where some boundary configuration is fixed. In particular, even in the case of graphs, this leads to new results that characterizes conditions on $\BH$ that guarantee connectedness and mixing properties when some variables are fixed to take particular values.

As a byproduct of our results, we obtain two applications. On the one hand, we establish a link with \emph{strong spatial mixing} (e.g., see \cite{1-dyer}) and \emph{topological strong spatial mixing} (introduced in \cite{1-briceno}). These two last properties have played an important role in the development of deterministic approximate counting algorithms. In this paper we address the following question: What fixed targets $\BH$ are suitable for both of these properties to hold for any $\BG$? On the other hand, we establish a connection with finite duality, which allows us to reprove the main theorem in \cite{MR2357493}. We hope that our work opens the possibility of developing new counting techniques based on this approach in a very general setting.

\subsection*{Organization} The paper is organized as follows: In Section \ref{section2}, we introduce most of the objects and terminology relevant to this work. In particular, we define relational structures, the natural maps and operations on them, and some useful constructions. In addition, we introduce the three relevant properties around homomorphisms: dismantlability, connectedness, and mixing. Next, in Section \ref{section3}, we state the two main theorems of this paper, Theorem \ref{the:mainsimple} and its refinement, Theorem \ref{the:main}, and we illustrate how these two theorems relate to the work of Brightwell and Winkler and generalizes it. Next, in Section \ref{section4}, we prove all the equivalences which constitute Theorem \ref{the:main}. In Section \ref{section5}, we define Gibbs measures on sets of homomorphisms and explore the consequences of our results in relation to spatial mixing properties of such measures. Finally, in Section \ref{section6}, we establish a meaningful connection between our results, topological strong spatial mixing, and finite duality of relational structures.

\section{Preliminaries}
\label{section2}

Let $H$ be a countable (finite or denumerable) set and $\arity$ a positive integer. The set of $\arity$-tuples over $H$ is denoted by $H^k$. A ($k$-ary) {\bf relation} $R$ over $H$ is a subset $R \subseteq H^k$. The elements of a relation $R$ will be denoted in boldface, e.g., $\ba$, $\bb$, etc., and $a_i$ will denote the $i$th entry of $\ba$ for $1 \leq i \leq \arity$.

Given another countable set $G$ and a map $\phi:G\rightarrow H$, for a $\arity$-tuple $\ba$ over $G$ we shall use $\phi(\ba)$ to denote the $\arity$-tuple over $H$ obtained after applying $\phi$ to $\ba$ componentwise. If $V \subseteq G$, we will denote by $\left.\phi\right|_V$ the restriction of $\phi$ to $V$. Furthermore, if $\psi$ is another map with domain $H$, we shall use $\psi \circ \phi$ to denote the composition of $\psi$ with $\phi$, i.e., the map $x\mapsto \psi(\phi(x))$.

A {\bf signature} $\tau$ is a collection of relation symbols $R$, each of them with an associated arity. For a given signature $\tau$, a {\bf relational structure} (with signature $\tau$)---or simply, a {\bf $\tau$-structure}---$\BH$ consists of a countable set $H$ called the {\bf universe} of $\BH$ and a relation $R(\BH) \subseteq H^{k}$ for each $R \in \tau$. Thus relation symbols serve as markers that allow one to associate relations in different structures with the same signature (e.g., see \cite{Hodges97}).

We shall use the same capital letter to denote the universe of a $\tau$-structure, e.g., $H$ is the universe of $\BH$. We will usually consider $\tau$ to be a fixed signature, and $\BG$ and $\BH$ to be $\tau$-structures with universes $G$ and $H$, respectively. 

The {\bf degree} of an element in a relational structure is defined as the number of tuples in which it occurs. A relational structure is {\bf locally finite} if every element in its universe has finite degree and is said to be {\bf finite} if its universe is finite.

\begin{remark}
{\normalfont
A digraph $\BG$ (with self-loops allowed) is a very particular case of a relational structure, where the signature $\tau$ consists of a unique relation symbol $E$ of arity $2$. Moreover, graphs correspond to the digraph case where the edge-set $E(\BG)$ is a symmetric relation.}
\end{remark}

A map $\phi:G\rightarrow H$ is said to be a {\bf homomorphism} from $\BG$ to $\BH$ if, for every relation symbol $R\in\tau$,
$$
\ba \in R(\BG) \implies \phi(\ba)\in R(\BH).
$$

We will denote by $\Hom(\BG,\BH)$ the set of all homomorphisms from $\BG$ to $\BH$.

\begin{example}  
\label{exmp:sft}
{\normalfont 
A particular example of CSPs that cannot be represented in the setting of Brightwell and Winkler (that is, as homomorphisms of graphs) is the case of \emph{shifts of finite type}, a fundamental object in dynamical systems and probability \cite{lind1995,lind2002,marcus2012}.

Given a positive integer $d$, consider the signature $\tau = \{R_1,\dots,R_d\}$, where $R_i$ is a $2$-ary relation for all $1 \leq i \leq d$. We consider two $\tau$-structures $\BG$ and $\BH$. Here, $\BG$ will be an infinite relational structure with universe $G = \mathbb{Z}^d$ and relations $R_i(\BG)$, $1 \leq i \leq d$, representing the usual $d$-dimensional hypercubic lattice and the adjacency of pairs of elements in it. On the other hand, $\BH$ will be a finite relational structure with universe $H$ and $R_i(\BH)$, $1 \leq i \leq d$, will represent pairs of ``colors'' from $H$ that are allowed to be adjacent in the canonical $i$th direction of the lattice. Then, $X = \Hom(\BG,\BH)$ is known as a {\bf $d$-dimensional nearest-neighbor shift of finite type}, a set of colorings of $\mathbb{Z}^d$ with not necessarily \emph{isotropic} adjacency rules (i.e., we do not need to have the same restrictions in every direction), and any such object can be represented in this way (see also \cite{1-schraudner} for a similar approach). It is well-known that every tiling of $\mathbb{Z}^d$ can be represented as a nearest-neighbor shift of finite type and vice versa \cite{schmidt2001multi}.

More generally, it can be checked that for any shift of finite type $X$ defined on a countable group $\Gamma$ (see \cite{ceccherini2010cellular} for an introduction to the subject), i.e., not necessarily nearest-neighbor nor restricted to $\Gamma = \mathbb{Z}^d$, there exists a signature $\tau$ and $\tau$-structures $\BG$ and $\BH$ such that $X = \Hom(\BG,\BH)$.
}
\end{example}

\begin{example}
{\normalfont
Hypergraph homomorphisms can be also naturally encoded using relational structures. A  ${\bf hypergraph}$ is a pair $(V,E)$ where $V$ is a set of elements called \emph{nodes} and $E$ is a collections of subsets of $V$ called \emph{hyperedges}. If $H=(V,E)$, $H'=(V',E')$ are hypergraphs, then a mapping $f:V\rightarrow V'$ is a homomorphism from $H$ to $H'$ if for every hyperedge $e\in E$, $\bigcup_{v\in e} \{f(v)\}$ is a hyperedge of $E'$. Now, we can encode $H$ and $H'$ as relational structures over a signature $\tau$ containing a unique relation symbol $R$ of arity $m$, where $m$ is the maximum size of any hyperedge in $E\cup E'$. In particular, $H$ is encoded by the relational structure $\BH$ with universe $V$ and
$$R(\BH)=\{(v_1,\dots,v_m) \mid \{v_1,\dots,v_m\}\in E\}.$$

Analogously, $H'$ can be encoded in a similar way. It is not difficult to see that every mapping $f:V\rightarrow V'$  is a homomorphism from $H$ to $H'$ if and only if it is a homomorphism of their associated relational structures.
}
\end{example}

A relational structure $\BJ$ is a {\bf substructure} of $\BH$ if $J \subseteq H$ and, for every relation symbol $R \in \tau$, we have that $R(\BJ)\subseteq R(\BH)$. Furthermore, if for every $\arity$-ary $R\in\tau$, we have that $R(\BJ)= R(\BH)\cap J^k$, then we say that $\BJ$ is the substructure of $\BH$ {\bf induced by} $J$. If $J \subseteq H$ and $\phi:H \rightarrow J$ is a homomorphism acting as the identity on $J$, then $\phi$ is said to be a {\bf retraction}.

The {\bf product} of $\BH_1$ and $\BH_2$, denoted $\BH_1\times\BH_2$, is the $\tau$-structure with universe $H_1 \times H_2$ where, for every $\arity$-ary relation symbol $R \in \tau$, we have that $R(\BH_1 \times \BH_2)$ consists of all tuples $((a_1,b_1),\dots,(a_k,b_k))$ with $(a_1,\dots,a_k)\in R(\BH_1)$ and $(b_1,\dots,b_k)\in R(\BH_2)$. We shall denote by $\BH^2$ the product $\BH\times \BH$. The {\bf projections} $\pi_1,\pi_2:H^2\rightarrow H$ are the maps $(a,b) \mapsto a$ and $(a,b) \mapsto b$, respectively, for $(a,b)\in H^2$. An element $(a,b)$ of $H^2$ is \emph{diagonal} if $a=b$. The {\bf diagonal set} of $H^2$, denoted $\diagonal(H^2)$, is the set of its diagonal elements. Similarly, the {\bf diagonal structure} of $\BH^2$, denoted $\diagonal(\BH^2)$, is the substructure of $\BH^2$ induced by $\diagonal(H^2)$. A substructure $\BK$ of $\BH^2$ is {\bf symmetric} whenever $(a,b) \in K$ if and only if $(b,a) \in K$. Notice that $\BH^2$ is always symmetric.

In this paper, we will study properties of $\BH$ and how they relate to other properties of $\Hom(\BG,\BH)$ for arbitrary $\BG$. We mainly consider three families of properties, namely, \emph{dismantling} of $\BH$, \emph{connectedness} of some particular graphs with vertex set $\Hom(\BG,\BH)$, and \emph{mixing} properties of $\Hom(\BG,\BH)$. By the end of the article, we also show how all these properties are related to a fourth one, namely, \emph{finite duality}.

\subsection{Dismantling}

Let $\BH$ be a $\tau$-structure and let $a,b$ be elements in its universe $H$. We say that $b$ {\bf dominates} $a$ (in $\BH$) if for every $\arity$-ary $ R\in\tau$, any $i \in \{1,\dots,k\}$, and any $(a_1,\dots,a_k) \in  R(\BH)$ with $a_i = a$, we also have that
$$
(a_1,\dots,a_{i-1},b,a_{i+1},\dots,a_k)\in  R(\BH).
$$

Additionally, if $a\neq b$, then we say that $a$ is {\bf dominated} (in $\BH$).

A sequence of $\tau$-structures $\BJ_0,\dots,\BJ_{\ell}$ is a {\bf dismantling sequence} if for every $0\leq j<\ell$ there exist $a_j,b_j \in J_j$ such that $b_j$ dominates $a_j$ in $\BJ_j$, and $\BJ_{j+1}$ is the substructure of $\BJ_j$ induced by $J_j\setminus\{a_j\}$. In this case, we say that $\BJ_0$ {\bf dismantles to} $\BJ_{\ell}$. We can alternatively denote a dismantling sequence by giving the initial $\tau$-structure $\BJ_0$ and the sequence of elements $a_0,\dots,a_{\ell-1}$. We say that $\BH$ is {\bf dismantlable} if it dismantles to a $\tau$-structure such that its universe is a singleton.

Note that for every $0\leq j <\ell$ there is a natural retraction $r_j$ from $\BJ_j$ to $\BJ_{j+1}$, where $r_j$ maps $a_j$ to $b_j$ and acts as the identity elsewhere. We call such retractions a {\bf fold}. By successive composition, one can define a retraction (namely, $r_{j'-1} \circ \cdots \circ r_j$) from $\BJ_j$ to $\BJ_{j'}$ for every $j \leq j'$.

\begin{example}
{\normalfont 
In the context of $d$-dimensional nearest-neighbor SFTs from Example \ref{exmp:sft}, suppose that $d=2$ and consider $\tau = \{R_1,R_2\}$ with $R_i$ a $2$-ary relation symbol for $i=1,2$. Then, we can consider $\tau$-structure $\BG$ and $\BH$ such that
\begin{itemize}
\item the universe of $\BG$ is $\mathbb{Z}^2$,
\begin{align*}
R_1(\mathbb{G})	&	= \{(g, g + (1,0)): g \in \mathbb{Z}^2\}, \text{ and}\\
R_2(\mathbb{G})	&	= \{(g, g + (0,1)): g \in \mathbb{Z}^2\};
\end{align*}
\item the universe of $\BH$ is $H = \{a,b,c\}$,
\begin{align*}
R_1(\BH) &= \{(a,a),(a,b),(b,a),(b,b), (b,c),(c,b)\}, \text{ and}\\
R_2(\BH) &=  \{(a,a),(a,b),(b,a),(b,b),(b,c),(c,a)\}.
\end{align*}
\end{itemize}

We can notice that $c$ folds to $b$, and $b$ folds to $a$. We will see that the fact that $\BH$ can be dismantled to a single element will imply that $\Hom(\BG,\BH)$ has special properties. We also notice that any $\Hom(\BG,\BH)$ can be seen as a \emph{Wang tiling} of $\mathbb{Z}^2$ and vice versa \cite{schmidt2001multi}.
} 
\end{example}

It is well known that if $\BH$ dismantles to some substructure $\BK$, then this dismantling can be found in a greedy manner. Formally, we have the following lemma.

\begin{lemma}[{\cite[Lemma 5.1]{MR2357493}}]
\label{le:greedy}
If $\BH$ dismantles to $\BK$ and $a \in H \setminus K$ is dominated in $\BH$, then the substructure of $\BH$ induced by $H \setminus \{a\}$ dismantles to $\BK$.
\end{lemma}

Let $J \subseteq H$. We say that $\BH$ is {\bf $J$-non-foldable} if every dominated element in $\BH$ belongs to $J$. When $J = \emptyset$, i.e., if no fold is possible at all, $J$-non-foldable relational structures are a generalization from the graph case of \emph{stiff graphs}, used in \cite{1-brightwell}. There, it is proved that if we start from a given graph, one will always reach the same stiff graph up to isomorphism when dismantling greedily \cite[Theorem 4.4]{1-brightwell}. It can be checked that in the context of relational structures the following is true: for every $J \subseteq H$, dismantling greedily elements in $H \setminus J$ will always lead to isomorphic relational structures (moreover, using an isomorphism that preserves $J$).

\subsection{Walks in relational structures}

We define a {\bf walk}\footnote{See for example \cite{Kun12}. Our notion of walk can be alternatively seen as a walk in the incidence multigraph \cite{MR2357493} of $\BH$.} $w$ in a $\tau$-structure $\BH$ to be a sequence
$$a_0,i_1,(R_1,\ba_1),j_1,a_1,\dots,a_{n-	1},i_n,(R_n,\ba_n),j_n,a_n$$
for some $n \geq 0$, such that, for all $1 \leq \ell \leq n$,
\begin{itemize}
\item $R_\ell \in \tau$, $\ba_\ell \in R_\ell(\BH)$, $i_\ell \neq j_\ell$, and 
\item $a_{\ell-1} = \ba_{\ell}[i_{\ell}]$ and $a_{\ell} = \ba_{\ell}[j_\ell]$.
\end{itemize}

In this case, we will say that $w$ {\bf joins} $a_0$ (the {\bf starting point}) and $a_n$ (the {\bf ending point}), and that the {\bf length} of the walk $w$ is $n$. Notice that if a walk $w$ joins $a_0$ and $a_n$, then there is another walk $w'$ that joins $a_n$ and $a_0$ obtained by just reversing the order of indices. The {\bf distance} $\dist(a,b)$ between two elements $a,b\in H$ is defined to be the smallest length among all the walks $w$ that join $a$ and $b$. The distance $\dist(V,W)$ between sets $V,W \subseteq H$ is defined to be the minimum distance between an element from $V$ and an element from $W$. If $V$ is a singleton $\{a\}$, we will just write $\dist(a,W)$ instead of $\dist(\{a\},W)$.

Note that the definition of walk above coincides with the standard definition of walk when $\BH$ is a graph. However, in the case of graphs it will be convenient to describe the walk merely as the list $a_0,\dots,a_n$ of its nodes, as usual.

A $\tau$-structure $\BH$ is {\bf connected} if there is a walk joining any pair of elements of its universe $H$ and a {\bf connected component} is any induced substructure that is connected and maximal in the sense of inclusion. A walk $w$ is a {\bf cycle} if $n>0$, the starting and ending nodes are the only repeated nodes, and for all $1 \leq \ell < \ell' \leq n$, we have that $(R_\ell,\ba_\ell)\neq (R_{\ell'},\ba_{\ell'})$. A $\tau$-structure $\BT$ is a {\bf $\tau$-forest} if it has no cycles. 

If, additionally, it is connected then it is a {\bf $\tau$-tree}. Usually, $\tau$-trees are defined using the notion of incidence multigraph (e.g., see \cite{MR2357493}). It is easy to verify that the definition given here is equivalent.

\begin{remark}
We note here that some subtleties appear when encoding undirected graphs using relational structures (where tuples are ordered). It is usual  (for example as in the proof of Lemma \ref{lem:graphdism}) to encode the edge-set of an undirected graph as a binary symmetric relation since many properties (such as the existence of homomorphisms) are preserved under this encoding. However, some of the connectedness notions (e.g., cycle, tree, etc.) introduced in this section do not correspond to the standard graph-theoretical notions when graphs are encoded as relational structures in this way. Instead, if one encodes the edge set of a graph with a binary relation containing one tuple per edge (ordering the nodes in any arbitrary way) then the concepts introduced in this section correspond, over graphs, to the usual graph-theoretic meaning of the term.
\end{remark}

\subsection{Forest of walks}
\label{forest}
Given a $\tau$-structure $\BH$, we proceed to define a new $\tau$-structure $\BT_\BH$. The universe $T_\BH$ of $\BT_\BH$ consists of all the walks $w$ in $\BH$. For a $\arity$-ary $R \in \tau$, we define $R(\BT_\BH)$ as follows: for all $\ba=(a_1,\dots,a_k) \in R(\BH)$, for all $1 \leq i \leq \arity$, and for all walks $w$ ending in $a_i$, we include in $R(\BT_\BH)$ the tuple $(w_1,\dots,w_{i-1},w,w_{i+1},\dots,w_\arity)$, where $w_j$, $j \neq i$, is the walk obtained from $w$ by extending it with $i,(R,\ba),j,a_j$. 

We note that $\BT_{\BH}$ does not contain cycles and has exactly $|H|$ connected components, i.e., $|H|$ $\tau$-trees in correspondence with the possible starting element in $H$ for a walk in $T_\BH$. It is easy to check that for every substructure $\BI$ of $\BH$, the $\tau$-structure $\BT_\BI$ is a substructure of $\BT_\BH$.

\begin{remark}
{\normalfont
If $\BH$ is connected and we consider a slight modification of this previous definition, where the walks are asked to be non-backtracking \cite{Kun12}
 (i.e., for every $1 \leq \ell < n$,  we have that either $i_\ell \neq j_{\ell+1}$, or $j_\ell \neq i_{\ell+1}$, or $(R_\ell,\ba_\ell) \neq (R_{\ell+1},\ba_{\ell+1})$) then we obtain that each connected component of the resulting $\tau$-structure corresponds to the \emph{universal covering tree} of $\BH$ \cite{Kozik16,Kun12}. In particular, they are all the same up to isomorphism.
}
\end{remark}

Note that, by construction, the map $\rho_{\BH}: T_\BH \to H$ that sends every walk $w$ in $T_{\BH}$ to its ending point, that from now on we refer as the {\bf label map}, defines a homomorphism from $\BT_{\BH}$ to $\BH$. Furthermore,

\begin{lemma}
\label{le:universaltree}
Assume that $\BH$ is $J$-non-foldable for some $J\subseteq H$ and let $U$ be a cofinite subset of $T_{\BH}$ containing $\rho_\BH^{-1}(J)$. Then, every homomorphism in $\Hom(\BT_{\BH},\BH)$ that agrees with $\rho_{\BH}$ in $U$ is identical to $\rho_{\BH}$.
\end{lemma}

\begin{proof}
Given $n \geq 0$, let $W_n$ be the set of walks of length at least $n$ in $\BH$ (notice that $W_n \subseteq W_{n-1}$ and $W_0 = T_\BH$). We shall show that any $\rho'\in\Hom(\BT_{\BH},\BH)$ that agrees with $\rho_{\BH}$ in $W_n \cup \rho_\BH^{-1}(J)$ for arbitrary $n$ also agrees with $\rho_\BH$ in $W_{n-1}$ . Let $w$ be any walk of length $n-1$ and let $a$ be its ending point. We first show that $\rho'(w)$ dominates $a$ in $\BH$. Indeed, let $R\in \tau$ and let $\ba=(a_1,\dots,a_k)\in R(\BH)$, where $a$ appears, say, in the $i$th coordinate. By construction, $R(\BT_{\BH})$ contains  the tuple $\bw=(w_1,\dots,w_{i-1},w,w_{i+1},\dots,w_k)$, where for every $j \neq i$, $w_j$ is obtained by concatenating $i, (R,\ba), j, a_j$ at the end of $w$. Since $w_j$ has length $n$ for every $j\neq i$, it follows by assumption that $\rho'(w_j)=a_j$. That is, $\rho'(\bw)$ (which must be a tuple in $R(\BH)$) is obtained by replacing, in $\ba$, $a_i$ by $\rho'(w)$.

Hence, we have shown that $\rho'(w)$ dominates $a$ in $\BH$. Since $\BH$ is $J$-non-foldable it follows that either $\rho'(w)=a$ (and, hence, $\rho'(w)=\rho_\BH(w)$) or $a\in J$ (and, hence, $\rho'(w)=\rho_\BH(w)$ since $w\in \rho_\BH^{-1}(J)$). To conclude the proof it is only necessary to observe that, since $U$ is a cofinite set containing $\rho_\BH^{-1}(J)$, it follows that any homomorphism that agrees with $\rho_{\BH}$ in $U$, agrees as well in $W_n \cup \rho_\BH^{-1}(J)$ for sufficiently large $n$.
\end{proof}

\subsection{Graphs of homomorphisms}
\label{subsection:graphHom}

Let $\BG$ and $\BH$ be $\tau$-structures and suppose that $\BH$ is finite. We define two different kinds of graphs with vertex set $\Hom(\BG,\BH)$.

The first notion has been heavily studied, from an algorithmic perspective, in the context of the so-called \emph{CSP reconfiguration problem} (see \cite{Hatanaka18} and the references therein) and also from a structural point of view in the special case when $\BG$ and $\BH$ are graphs \cite{1-brightwell,1-wrochna}. We define $C(\BG,\BH)$ as the (reflexive) graph with vertex set $\Hom(\BG,\BH)$ such that for every $\phi,\psi \in\Hom(\BG,\BH)$, $\phi$ and $\psi$ are adjacent if and only if $\phi$ and $\psi$ differ in at most one value, i.e., there exists at most one $x \in G$ such that $\phi(x) \neq \psi(x)$. More generally, for any $n \geq 1$ we can define $C_n(\BG,\BH)$ on $\Hom(\BG,\BH)$ by declaring $\phi$ and $\psi$ adjacent if they differ in at most $n$ values (in particular, $C(\BG,\BH) = C_1(\BG,\BH)$).

A second notion of graph of homomorphisms appears in \cite{MR2357493} and uses the notion of \emph{links}. The {\bf $1$-link} $\BL$ (with signature $\tau$) is the $\tau$-structure with universe $\{0,1\}$, where $R(\BL) = \{0,1\}^k$ for every $\arity$-ary $R\in\tau$. Define a (di)graph $L(\BG,\BH)$ with vertex set $\Hom(\BG,\BH)$ as follows: set $\phi \rightarrow \psi$---i.e., a directed edge starting from $\phi$ and ending in $\psi$---if for any $\arity$-ary $R\in\tau$ and any $(x_1,\dots,x_k) \in R(\BG)$, we have that $(\gamma_1(x_1),\dots,\gamma_k(x_k)) \in R(\BH)$ whenever $\gamma_1,\dots,\gamma_k\in \{\phi,\psi\}$. Alternatively, one can say that $\phi$ and $\psi$ are joined by a directed edge if there exists a homomorphism from $\BL$ to $\Hom(\BG,\BH)$ (see \cite[Section 5.2]{MR2357493}), mapping $0$ to $\phi$ and $1$ to $\psi$. Notice that the symmetry in the definition of $1$-link implies that $L(\BG,\BH)$ is, in fact, an undirected graph.

We say that a graph of homomorphisms (i.e., $C(\BG,\BH)$, $C_n(\BG,\BH)$, $L(\BG,\BH)$) is \emph{connected} if for every pair of maps $\phi,\psi \in \Hom(\BG,\BH)$ agreeing on all but finitely many elements there exists a walk that joins them. Notice that if $\BG$ is finite, this coincides with the usual notion of connectedness in graph theory.
 
Clearly, $C_n(\BG,\BH)$ is a subgraph of $C_{n+1}(\BG,\BH)$. In contrast, $C_n(\BG,\BH)$ and $L(\BG,\BH)$ are not included in one another in general. However, we will establish (see Lemma \ref{le:conn}) a meaningful relationship between both of them, by characterizing the connectivity properties of one in terms of the other.

Note that there is a one-to-one correspondence between the elements in $\Hom(\BL \times \BG,\BH)$ and the edges of $L(\BG,\BH)$. More generally, for $\ell \geq 1$ we define the {\bf $\ell$-link} $\BL_\ell$ (with signature $\tau$) as the $\tau$-structure with universe $\{0,1,\dots,\ell\}$, where $R(\BL_\ell) = \cup_{i=0}^{\ell-1}\{i,i+1\}^k$, for every $\arity$-ary $R\in\tau$. In other words, the $\ell$-link is a sequence of $1$-links with their endpoints identified. Then the following result
is immediate:

\begin{lemma}
\label{le:links}
For every map $\phi:\{0,1,\dots,\ell\}\times G\rightarrow H$ and every $1\leq i\leq \ell$, let $\phi(i):G\rightarrow H$ be the map defined by $\phi(i)(x) \mapsto \phi(i,x)$ for $x\in G$. Then, $\phi\in\Hom(\BL_\ell \times \BG,\BH)$ if and only if $\phi(0),\dots,\phi(\ell)$ is a {walk} in $L(\BG,\BH)$.
\end{lemma}

\subsection{Mixing properties}

Given $\tau$-structures $\BG$ and $\BH$, it is useful to study properties in $\Hom(\BG,\BH)$ that allow us to \emph{glue together} partially defined homomorphisms. This kind of properties are usually referred in the literature as \emph{mixing properties}.

We say that $\Hom(\BG,\BH)$ is {\bf $(V,W)$-mixing} for $V,W \subseteq G$, if for every $\phi,\psi \in \Hom(\BG,\BH)$, there exists a map $\gamma \in \Hom(\BG,\BH)$ that agrees with $\phi$ on $V$ and agrees with $\psi$ on $W$. Given $\gap \geq 0$, we say that $\Hom(\BG,\BH)$ is {\bf strongly irreducible with gap $\gap$} if it is $(V,W)$-mixing for every $V,W$ such that $\dist(V,W) \geq \gap$. We say that $\Hom(\BG,\BH)$ is {\bf strongly irreducible} if it is strongly irreducible with gap $\gap$ for some $\gap$.

A strengthening of strong irreducibility is the following property, introduced in \cite{1-briceno}. Given $\gap \geq 0$, we say that $\Hom(\BG,\BH)$ is {\bf topologically strong spatial mixing (TSSM) with gap $\gap$} if for every $V,W,S \subseteq G$ such that $\dist(V,W) \geq \gap$ and for all $\phi,\psi \in \Hom(\BG,\BH)$ that agree on $S$, there exists $\gamma \in \Hom(\BG,\BH)$ that agrees with $\phi$ on $V \cup S$ and agrees with $\psi$ on $S \cup W$. We say that $\Hom(\BG,\BH)$ is {\bf topologically strong spatial mixing} if it is TSSM with gap $\gap$ for some $\gap$.

Clearly, $\Hom(\BG,\BH)$ is TSSM only if $\Hom(\BG,\BH)$ is strongly irreducible but not vice versa (see \cite{1-briceno, 2-briceno} for some counterexamples).

An antithesis of having good mixing properties is the existence of homomorphisms which are \emph{frozen}. We say that $\phi \in \Hom(\BG,\BH)$ is a {\bf frozen homomorphism} if for any cofinite set $U \subseteq G$, the only homomorphism $\psi \in \Hom(\BG,\BH)$ such that $\left.\psi\right|_U = \left.\phi\right|_U$ is $\psi = \phi$ itself. Notice that Lemma \ref{le:universaltree}, when $J = \emptyset$, says that $\rho_{\BH}$ is a frozen homomorphism in $\Hom(\BT_{\BH},\BH)$.

\section{Main theorems}
\label{section3}

In this section we present the two main theorems of our work, which characterize in several ways a special class of relational structures. Both theorems consist of a generalization of some of the equivalences characterizing \emph{dismantlable graphs} that appear in \cite[Theorem 4.1]{1-brightwell}---which were developed only for the case of graphs---in two directions. First, Theorem \ref{the:mainsimple} (or the \emph{simple theorem}) extends \cite{1-brightwell} from graphs to arbitrary relational structures. Second, Theorem \ref{the:main} (or the \emph{refined theorem}), shows how the equivalences in Theorem \ref{the:mainsimple} can be rephrased in terms of stronger properties with respect to special subsets of the universe of the given relational structure.

\subsection{The case of graphs} 

The following theorem is a rephrasing of the equivalences that appear in \cite[Theorem 4.1]{1-brightwell} which are relevant to us. We will use this as a prototypical example of the kind of results that we are aiming for, where we split the properties in 3 main categories: (A) dismantlability, (B) connectedness, and (C) mixing.

We suppose that graphs are allowed to have loops. Notice that a loopless graph cannot be dismantlable.

\begin{theorem}[{\cite[Theorem 4.1]{1-brightwell}}]
\label{the:graph}
Let $\BH$ be a finite graph. The following are equivalent:
\begin{enumerate}
\item[(A)] $\BH$ is dismantlable;
\item[(B)] $C(\BG,\BH)$ is connected for every finite graph $\BG$; and
\item[(C)] there exists $\gap \geq 0$ such that $\Hom(\BG,\BH)$ is strongly irreducible with gap $\gap$ for every graph $\BG$.
\end{enumerate}
\end{theorem}

\subsection{First theorem: A parallel with the graph case}

The following theorem shows that different dismantling, connectedness, and mixing notions are equivalent. It can be seen as a generalization of Theorem \ref{the:graph} to relational structures.

\begin{theorem}
\label{the:mainsimple}
Let $\BH$ be a finite $\tau$-structure with universe $H$. Then the following are equivalent:
\begin{enumerate}
\item[(A1s)] $\BH$ dismantles to a substructure $\BI$ such that $\BI^2$ dismantles to its diagonal\label{itemdismantlesrestricted};
\item[(A2s)] $\BH^2$ dismantles to a substructure of its diagonal;
\item[(B1s)] $C(\BG,\BH)$ is connected for every locally finite $\tau$-structure $\BG$\label{itemC1};
\item[(B2s)] there exists some $\diff \geq 1$ such that $C_\diff (\BG,\BH)$ is connected for every finite $\tau$-structure $\BG$\label{itemCexists}; \item[(B3s)] $C(\BL \times \BH^2,\BH)$ is connected\label{itemCparticular};
\item[(B4s)] $L(\BG,\BH)$ is connected for every finite $\tau$-structure $\BG$\label{itemLforall};
\item[(B5s)] the projections $\pi_1$ and $\pi_2$ are connected in $L(\BH^2,\BH)$\label{itemLparticular};
\item[(C1s)] there exists $\gap \geq 0$ such that $\Hom(\BG,\BH)$ is strongly irreducible with gap $\gap$ for every $\tau$-structure $\BG$; and
\item[(C2s)] there exists $\gap \geq 0$ such that $\Hom(\BT_{\BH^2},\BH)$ is $(\{x\},W)$-mixing for all $x \in T_{\BH^2}$ and $W \subseteq T_{\BH^2}$ with $\dist(x,W) \geq \gap$.
\end{enumerate}
\end{theorem}

As a consequence of our results we have the following lemma.

\begin{lemma}
\label{lem:graphdism}
A graph $\BH$ is dismantlable if and only if $\BH^2$ dismantles to a subset of its diagonal.
\end{lemma}

\begin{proof}
In Theorem \ref{the:mainsimple}, we prove that, for a finite $\tau$-structure $\BH$, we have that $\BH^2$ dismantles to a substructure of its diagonal if and only if there exists $\gap \geq 0$ such that $\Hom(\BG,\BH)$ is strongly irreducible with gap $\gap$ for all $\tau$-structures $\BG$. In particular, this applies if $\tau = \{E\}$, the usual binary relation of adjacency in graphs. Therefore, by Theorem \ref{the:graph}, these two properties are also equivalent to $\BH$ being dismantlable, and we conclude.  
\end{proof}

In other words, thanks to Lemma \ref{lem:graphdism}, at least in the realm of graphs, we can freely replace ``dismantlable'' by ``the square dismantles to a substructure of its diagonal''.

Still, it is not necessary to invoke Theorem \ref{the:mainsimple} and Theorem \ref{the:graph} in order to prove Lemma \ref{lem:graphdism}, as this can be directly proved even in a more general setting. 

A relation $R$ of arity $k$ is {\em symmetric}\footnote{Not to be confused with the notion of symmetry introduced in Section \ref{section2}.} if for every tuple $(t_1,\dots,t_k)\in R$, and every permutation $\sigma:\{1,\dots,k\}\rightarrow\{1,\dots,k\}$, 
$(t_{\sigma(1)},\dots,t_{\sigma(k)})\in R$.

\begin{lemma}
\label{le:symmetricstructures}
Assume that $\tau=\{R\}$ contains a unique relation symbol and let $\BH$ be a relational structure such that $R(\BH)$ is symmetric.
Then $\BH$ is dismantlable if and only if $\BH^2$ dismantles to a subset of its diagonal.
\end{lemma}
\begin{proof}
It is immediate that if $a$ is dominated in $\BH$, then $\BH^2$ dismantles to $\BK^2$, where $\BK$ is the substructure of $\BH$ induced by $H\setminus\{a\}$, which implies the ``only if'' part of the statement.
For the converse, assume that $\BH^2$ dismantles to a subset of its diagonal and let $\BK$ be a $\emptyset$-non-foldable structure obtained from $\BH$ by a sequence of folds until no further folding is possible. Then $\BK^2$ dismantles to its diagonal as well (this fact is easy to see but to avoid repeating arguments given elsewhere we refer here to Remark \ref{pro:effective}). Assume, towards a contradiction, that the universe of $\BK$ is not a singleton. In consequence, the universe of $\BK^2$ is not a singleton either and there must exist some element $(a,b)$ dominated in $\BK^2$. Let $(a',b')$ be the element in $K^2$ dominating $(a,b)$ and assume that $a'\neq a$ (the case $b'\neq b$ is analogous). We shall prove that $a'$ dominates $a$ in $\BK$, contradicting the fact that $\BK$ is $\emptyset$-non-foldable. Indeed, let $(a_1,\dots,a_k)$ be any tuple in $R(\BK)$
and let $i\in\{1,\dots,k\}$ with $a_i=a$. Note that $b$ is not an isolated element (meaning it appears in some tuple of $R(\BK)$) since this would contradict the fact that $\BK$ is $\emptyset$-non-foldable. Note that $R(\BK)$ is symmetric as well, implying that there exists a tuple $(b_1,\dots,b_k)\in R(\BK)$ with 
$b_i=b$. Since $(a,b)$ is dominated by $(a',b')$ in $\BK^2$, it follows that $R(\BK^2)$ contains the tuple obtained after replacing $(a_i,b_i)$ by $(a',b')$ in  $((a_1,b_1),\dots,(a_k,b_k))$. This implies that the tuple obtained replacing $a_i$ by $a'$ in $(a_1,\dots,a_k)$ belongs to $R(\BK)$.
\end{proof}

\begin{question}
{\normalfont
While it is always true that if $\mathbb{H}$ dismantles to a single element, then $\mathbb{H}^2$ dismantles to a substructure of the diagonal
it is important to notice that the equivalence between ``dismantlable'' and ``the square dismantles to a substructure of its diagonal'' is not true for general relational structures. For example, given $\tau = \{R\}$ for $R$ a binary relation symbol, we can take $\BH$ such that $H = \{0,1\}$ and $R(\BH) = \{(0,1)\}$. Then, $\BH$ is not dismantlable, but $\BH^2$ dismantles to its diagonal.structures. It is an interesting open question to determine for which relational structures, besides those consisting of a unique symmetric relation, both notions coincide.}
\end{question}

In Section \ref{section4} we shall prove Theorem \ref{the:mainsimple}. Indeed, we shall prove a refinement of it (Theorem~\ref{the:main}). This refined theorem, which we believe is interesting on its own, is motivated by the fact that, sometimes, it is natural---particularly in the context of statistical physics---to work by forcing each particle of some subset to each take a particular spin and work with the remaining ones. For example, this is a common scenario when the particles in the boundary of a given set in a lattice are fixed to take particular spins and we want to study the distribution of spins in the interior of the set, conditioned on such boundary configuration. These ideas inspired the refined version, which can be regarded as the study of \emph{boundary long range actions}, i.e., long range action phenomena where some boundary configuration is fixed, very similar to the concept of \emph{boundary phase transition} in relation to phase transitions \cite{1-martinelli}. In order to state this stronger version, we need the following definitions.

\subsection{Some refined definitions}

Let $\BH$ and $\BG$ be $\tau$-structures where $\BG$ is possibly infinite. Let $\phi_1,\dots,\phi_\seq$ be a sequence of homomorphisms in $\Hom(\BG,\BH)$ and let $J \subseteq H$. We say that $\phi_1,\dots,\phi_\seq$ is {\bf $J$-preserving} if for every $x\in G$ such that $\phi_1(x)=\phi_\seq(x)=a\in J$,  we have that $\phi_i(x)=a$ for every $1\leq i\leq \seq$.

A {\bf $J$-walk} is a $J$-preserving walk. Furthermore, we say that a graph of homomorphisms is {\bf $J$-connected} if for every pair of maps $\phi,\psi \in \Hom(\BG,\BH)$ agreeing on all but finitely many elements there exists a $J$-walk that joins them. Notice that if $J = \emptyset$, the definition of $J$-connectedness coincides with the definition of connectedness introduced in Section \ref{subsection:graphHom}.

Given $J \subseteq H$, we say that $\Hom(\BG,\BH)$ is {\bf $(V,W)$-mixing with respect to $J$} if $\Hom(\BG,\BH)$ is $(V,W)$-mixing and the map $\gamma$ can be chosen so that $\gamma(x) = \phi(x)=\psi(x)$ for all $x \in (\phi,\psi)^{-1}(\Delta(J^2))$, where $(\phi,\psi)^{-1}$ denotes the inverse of the map $(\phi,\psi): \BG \to \BH^2$ given by $x \mapsto (\phi(x),\psi(x))$; i.e., the map $\gamma$ coincides with $\phi$ in $V$, with $\psi$ in $W$, and with both of them for $x$ such that $\phi(x)=\psi(x) \in J$. We say that $\Hom(\BG,\BH)$ is {\bf strongly $J$-irreducible} with gap $\gap$ if it is $(V,W)$-mixing with respect to $J$ for all $V,W$ such that $\dist(V,W) \geq \gap$.

It is easy to check that $\Hom(\BG,\BH)$ is strongly $H$-irreducible with gap $\gap$ if and only if $\Hom(\BG,\BH)$ is TSSM with gap $\gap$. Indeed, the equivalence comes from the fact that the set $S$ in the definition of TSSM is always a subset of $(\phi,\psi)^{-1}(\Delta(H^2))$ and the strongly $H$-irreducible property is equivalent to the TSSM case where $S = (\phi,\psi)^{-1}(\Delta(H^2))$.

\subsection{Second theorem: A refinement}
   
\begin{theorem}\label{the:main}
Let $\BH$ be a finite $\tau$-structure with universe $H$ and let $J\subseteq H$. Then the following are equivalent:
\begin{enumerate}
\item[(\mylabel{itemdismantlesrestricted}{A1})] $\BH$ dismantles to a substructure $\BI$ whose universe contains $J$ and such that $\BI^2$ dismantles to its diagonal;
\item[(\mylabel{itemdismantles}{A2})] {$\BH^2$ dismantles to a substructure $\BK$ where its universe $K$ satisfies $\diagonal(J^2)\subseteq K\subseteq\diagonal(H^2)$;}
\item[(\mylabel{itemC1}{B1})] $C(\BG,\BH)$ is $J$-connected for every locally finite $\tau$-structure $\BG$;
\item[(\mylabel{itemCexists}{B2})] there exists some $\diff \geq 1$ such that $C_\diff (\BG,\BH)$ is $J$-connected for every finite $\tau$-structure $\BG$;    
\item[(\mylabel{itemCparticular}{B3})] $C(\BL \times \BH^2,\BH)$ is $J$-connected;
\item[(\mylabel{itemLforall}{B4})] $L(\BG,\BH)$ is $J$-connected for every finite $\tau$-structure $\BG$;
\item[(\mylabel{itemLparticular}{B5})] the projections $\pi_1$ and $\pi_2$ are $J$-connected in $L(\BH^2,\BH)$;
\item[(\mylabel{itemgraph}{C1})] there exists $\gap \geq 0$ such that $\Hom(\BG,\BH)$ is strongly $J$-irreducible with gap $\gap$ for every $\tau$-structure $\BG$; and
\item[(\mylabel{itemtree}{C2})] there exists $\gap \geq 0$ such that $\Hom(\BT_{\BH^2},\BH)$ is $(\{x\},W)$-mixing with respect to $J$ for all $x \in T_{\BH^2}$ and $W \subseteq T_\BH^2$  with $\dist(x,W) \geq \gap$.
\end{enumerate}
\end{theorem}

\begin{remark}
{\normalfont
To our knowledge, if $J \neq \emptyset$, Theorem \ref{the:main} has not been known before even in the graph case. In addition, notice that from Theorem \ref{the:main} it follows, by taking $J = H$, that $\BH^2$ dismantles to its full diagonal if and only if there exists $\gap \geq 0$ such that $\Hom(\BG,\BH)$ is TSSM with gap $\gap$ for every $\tau$-structure $\BG$.}
\end{remark}

As a byproduct of our results (in particular of Lemma \ref{le:path-dism} in Section \ref{section4}) we derive the following effective procedure to decide whether the statements of Theorem \ref{the:main} are satisfied.

\begin{remark}
\label{pro:effective}
{\normalfont
Let $\BH$ be a finite $\tau$-structure and let $J\subseteq H$. If there is a dismantling sequence (for $\BH$ and $J$) as in Theorem \ref{the:main}(\ref{itemdismantlesrestricted}), then there is one that can be obtained in the following greedy manner:
\begin{itemize}
\item First step. Starting from $\BH$, iteratively fold elements in $H \setminus J$ in any arbitrary way until no further dismantling is possible, obtaining a $J$-non-foldable $\tau$-structure $\BI$.
\item Second step. Starting from $\BI^2$, iteratively fold symmetric pairs $(a,b)$ and $(b,a)$ in an arbitrary way until no further dismantling is possible. 
\end{itemize}
}
\end{remark}

\section{Proofs}
\label{section4}
 
The following implications are immediate:
\begin{gather*}
(\ref{itemC1}) \Rightarrow (\ref{itemCexists}), (\ref{itemC1}) \Rightarrow (\ref{itemCparticular}), (\ref{itemLforall})\Rightarrow (\ref{itemLparticular}), \text{ and }(\ref{itemgraph})\Rightarrow (\ref{itemtree}).
\end{gather*}

The rest of the section contains several lemmas from which the remaining implications follow according the following table:
$$
\begin{array}{|c|c||c|c|}
\hline
 \text{Lemma \ref{le:conn}} 				&	(\ref{itemCexists}) \Rightarrow (\ref{itemLforall})					&	\text{Lemma \ref{le:retraction}}				&	(\ref{itemdismantlesrestricted}) \Rightarrow (\ref{itemdismantles})	\\
 \cline{3-4}
 									&	(\ref{itemCparticular}) \Rightarrow (\ref{itemLparticular})				&	\text{Lemma \ref{le:connectedness}} 			&	(\ref{itemdismantlesrestricted}) \Rightarrow (\ref{itemC1})			\\
\hline
\text{Lemma \ref{le:path-dism}} 				& 	(\ref{itemLparticular}) \Rightarrow (\ref{itemdismantlesrestricted})	&	\text{Lemma \ref{le:interpolate}} 			&	(\ref{itemdismantlesrestricted}) \Rightarrow (\ref{itemgraph})		\\
\hline
\text{Lemma \ref{le:distmantledrestricted}}		&	(\ref{itemdismantles}) \Rightarrow (\ref{itemdismantlesrestricted})	&	\text{Lemma \ref{le:free-tree}} 				&	(\ref{itemtree})\Rightarrow(\ref{itemdismantlesrestricted})			\\
\hline
\end{array}
$$

\begin{lemma}
\label{le:symmetric}
Let $\BI$ be a $J$-non-foldable $\tau$-structure and $\BK$ be a symmetric $\tau$-structure obtained by folding only non-diagonal elements of $\BI^2$ (note that $J^2 \subseteq K$). Assume that $\BK$ is minimal, i.e., $\BK$ has no proper substructure with the same property. Then, $\BK$ is $\diagonal(J^2)$-non-foldable.
\end{lemma}

\begin{proof}
Let $(a,b)$ be any dominated element in $\BK$. We first shall prove that $(a,b)$ is diagonal by contradiction. Let $(c,d) \neq (a,b)$ be an element dominating $(a,b)$ in $\BK$. Since $\BK$ is symmetric, it follows that both $(b,a)$ and $(d,c)$ are present in $K$ as well. If $(c,d)\neq (b,a)$, then we can fold both $(a,b)$ and $(b,a)$ in $\BK$, contradicting the minimality of $\BK$. Otherwise, it follows (as we shall prove straight away) that $(b,b)$ dominates both $(a,b)$ and $(b,a)$, contradicting again the minimality of $\BK$. 

We only need to prove that $(b,b)$ dominates $(a,b)$, as the other case is analogous. Let $R \in \tau$, let $((a_1,b_1),\dots,(a_k,b_k))$ be any tuple in $R(\BK)$, and let $j$ such that  $(a_j,b_j)=(a,b)$. Since $(b,a)$ dominates $(a,b)$,
$$
((a_1,b_1),\dots,(a_{j-1},b_{j-1}),(b,a),(a_{j+1},b_{j+1}),\dots,(a_k,b_k))\in  R(\BK).
$$

It follows that $(a_1,\dots,a_{j-1},b,a_{j+1},\dots,a_k)\in R(\BI)$. Since $(b_1,\dots,b_k)$ is also a tuple in $ R(\BI)$ and $b_j=b$, it follows that
$$
((a_1,b_1),\dots,(a_{j-1},b_{j-1}),(b,b),(a_{j+1},b_{j+1}),\dots,(a_k,b_k))\in R(\BK).
$$

To complete the proof, we shall show that $a$ is dominated in $\BI$ which, by the assumptions on $\BI$, implies that $a\in J$.
Let $(c,d)$ be any element dominating $(a,a)$ in $\BK$. Cleary, $c$ or $d$ is different from $a$, so assume, w.l.o.g., that $a\neq c$. We shall show that $c$ dominates $a$ in $\BI$. Indeed, let $ R\in\tau$, let $(a_1,\dots,a_k)$ be any tuple in $ R(\BI)$, and let $j$ such that $a_j=a$. Then,
$((a_1,a_1),\dots,(a_k,a_k))\in  R(\BK)$ and thus, $((a_1,a_1),\dots,(a_{j-1},a_{j-1}),(c,d),(a_{j+1},a_{j+1}),\dots,(a_k,a_k))\in  R(\BK)$ implies
that $(a_1,\dots,a_{j-1},c,a_{j+1}\dots,a_k)\in  R(\BI)$, concluding the proof.
\end{proof}

\begin{lemma}[{[(\ref{itemCexists}) $\Rightarrow$ (\ref{itemLforall}), (\ref{itemCparticular}) $\Rightarrow$ (\ref{itemLparticular})]}]
\label{le:conn} 
Given a finite $\tau$-structure $\BG$, if $C_\diff(\BL_\diff \times \BG,\BH)$ is $J$-connected for some $n \geq 1$, then $L(\BG,\BH)$ is $J$-connected.
\end{lemma}
 
\begin{proof}  
Let $\phi, \psi \in \Hom(\BG,\BH)$. We shall show that there is a $J$-walk in $L(\BG,\BH)$ from $\phi$ to $\psi$. By Lemma \ref{le:links}, the maps $h,h':\{0,\dots,\diff\}\times G\rightarrow H$, where $h(i)=\phi$ and $h'(i)=\psi$ for every $i\in\{0,\dots,\diff\}$, belong trivially to $\Hom(\BL_\diff \times \BG,\BH)$. Since we are assuming that $C_\diff(\BL_\diff \times \BG,\BH)$ is $J$-connected, it follows that there exists a $J$-walk $h=h^0,\dots,h^s=h'$ in $C_\diff(\BL_\diff \times \BG,\BH)$ joining them, for some $s \geq 0$. Proceeding by induction, we will construct a walk in $L(\BG,\BH)$ connecting $h^0(0)$ with $h^j(0)$ for every $j\in\{0,\dots,s\}$. The base case, $j=0$, is trivial. Now assume that the statement holds for some $j<s$. Since $h^j$ and $h^{j+1}$ differ in at most $\diff$ values, this implies that $(h^j(0),\dots,h^j(\diff))$ and $(h^{j+1}(0),\dots,h^{j+1}(\diff))$ must have an entry in common. Hence, let $i^*\in\{0,\dots,\diff\}$ be such that $h^j(i^*)=h^{j+1}(i^*)$. Since, by Lemma \ref{le:links}, there are walks from $h^j(0)$ to $h^j(i^*)$ and from $h^{j+1}(i^*)$ to $h^{j+1}(0)$, we are done.

Hence, we have shown that there is a walk in $L(\BG,\BH)$ connecting $h^0(0)=\phi$ and $h^s(0)=\psi$. It remains to show that the walk we have just constructed is $J$-preserving. Let $h^j(i)$, for $j \in\{0,\dots,s\}$ and $i\in\{0,\dots,\diff\}$, be any element in the walk, and let  $x\in G$ such that $\phi(x)=\psi(x)=a\in J$. Since $h^0(i)(x)=\phi(x)=a$, $h^s(i)(x)=\psi(x)=a$, and the walk $h^0,\dots,h^s$ is $J$-preserving, it follows that $h^j(i)=a$ as well.
\end{proof}

\begin{lemma}[{[(\ref{itemLparticular}) $\Rightarrow$ (\ref{itemdismantlesrestricted})]}]
\label{le:path-dism}
Suppose that the projections $\pi_1$ and $\pi_2$ are $J$-connected in $L(\BH^2,\BH)$. Let $\BI$ be any $J$-non-foldable structure obtained by dismantling $\BH$ and let $\BK$ be the symmetric $\diagonal(J^2)$-non-foldable structure given by Lemma~\ref{le:symmetric}. Then $\BK$ contains only diagonal elements.
\end{lemma}

\begin{proof}
Let $\BI$ and $\BK$ be as in the statement. It suffices to show that $\BK$ contains only diagonal elements. Let $r$ be the natural retraction from $\BH^2$ to $\BK$ obtained from a successive composition of folds. Let $\pi_1=h_1,\dots,h_\seq=\pi_2$ be a $J$-walk in $L(\BH^2,\BH)$ connecting $\pi_1$ and $\pi_2$ and consider the family of maps $\phi_1,\dots,\phi_{2\seq-1}:K\rightarrow K$, where
$$
\phi_i =
\begin{cases}
r(h_i,\pi_2)		&	\text{ if } i \leq \seq,	\\
r(\pi_2,h_{2\seq-i})	&	\text{ if } i >\seq.
\end{cases}
$$

It follows directly from the construction that every $\phi_i$ is an endomorphism of $\BK$ (i.e., a homomorphism from $\BK$ to $\BK$). If $\BK$ contains some non-diagonal element $(c,d)$, then it must also contain $(d,c)$. Hence, $\phi_1\neq \phi_{2\seq-1}$, since $\phi_1(c,d)=r(c,d)=(c,d)$ and $\phi_{2\seq-1}(c,d)=r(d,c)=(d,c)$. It follows that there exists some $i\leq\seq$ such that $\phi_1\neq \phi_i$ or there exists some $i\geq \seq$ such that $\phi_i\neq \phi_{2\seq-1}$. We shall consider only the first case since the proof for the other case is symmetric. Let $i$ be the minimum such that $\phi_1\neq \phi_i$. Also, let $(a,b)\in K$ with the property that $\phi_i(a,b) \neq \phi_1(a,b)=\phi_{i-1}(a,b)$. We shall prove that $\phi_i(a,b)$ dominates $(a,b)$ in $\BK$. Indeed, let $ R \in \tau$ of arity $\arity$, let $((a_1,b_1),\dots,(a_k,b_k))$ be any tuple in $ R(\BK)$, and let $j\in\{1,\dots,k\}$ be such that $(a_j,b_j)=(a,b)$. Consider now the $\arity$-tuple $(c_1,\dots,c_k)$, where $c_\ell = h_i(a_\ell,b_\ell) = h_i(a,b)$ for $j = \ell$ and $c_\ell = h_{i-1}(a_\ell,b_\ell)$ for $\ell \neq j$. Since $h_{i-1}$ and $h_i$ are adjacent in $L(\BH^2,\BH)$, it follows that $(c_1,\dots,c_k)\in R(\BH)$. Hence, the tuple $(d_1,\dots,d_k)$ with $d_\ell=r(c_\ell,b_\ell)$ belongs to $ R(\BK)$. Notice that, by definition, $d_j=r(h_i(a,b),b)=\phi_i(a,b)$ and $d_\ell = r(h_{i-1}(a_\ell,b_\ell),b_\ell) = \phi_{i-1}(a_\ell,b_\ell)=(a_\ell,b_\ell)$ for $\ell \neq j$.

We have just shown that $\phi_i(a,b)$ dominates $(a,b)$ in $\BK$. We know that $\BK$ is $\diagonal(J^2)$-non-foldable which does not lead yet to contradiction as it could be the case that $a=b$ and $a\in J$. In this case, note that $h_1(a,b)=h_\seq(a,b)=a$. Since $a\in J$ and $h_1,\dots,h_\seq$ is $J$-preserving, it follows that $h_j(a,b)=a$ for every $1\leq j\leq \seq$. This contradicts the fact that $\phi_i(a,b)\neq \phi_1(a,b)$.
\end{proof}

Note that Remark \ref{pro:effective} is a direct consequence of Lemma \ref{le:path-dism}.

\begin{lemma}[{[(\ref{itemdismantles}) $\Rightarrow$ (\ref{itemdismantlesrestricted})]}]
\label{le:distmantledrestricted}
If $\BH$ dismantles to a substructure $\BI$ whose universe contains $J$ and such that $\BI^2$ dismantles to its diagonal, then $\BH^2$ dismantles to a substructure $\BK$ where its universe $K$ satisfies $\diagonal(J^2)\subseteq K\subseteq\diagonal(H^2)$.
\end{lemma}

\begin{proof}
Assume that $\BH^2$ dismantles to a substructure $\BK$ where its universe $K$ satisfies $\diagonal(J^2)\subseteq K\subseteq\diagonal(H^2)$. We can assume that  $K\neq\diagonal(H^2)$ since otherwise there is nothing to prove.  Let $\BH^2=\BJ_0,\dots,\BJ_{\ell}=\BK$ the dismantling sequence and  $\BJ_i$ the last structure in the sequence whose domain contains $\diagonal(H^2)$. Consequently, $\BJ_i$ contains a diagonal element $(a,a)$ that is dominated by some other element $(b,c)$. Assume that $b\neq a$ (the other case is analogous). We claim that $b$ dominates $a$ in $\BH$. Indeed, let $R\in\tau$, let $(a_1,\dots,a_k)$ be any tuple in $R(\BH)$ and let $j$ such that $a_j=a$. Note that $R(\BJ_i)$ contains tuple $((a_1,a_1),\dots,(a_k,a_k))$ and, consequently, it also contains $((a_1,a_1),\dots,(a_{j-1},a_{j-1}),(b,c),(a_{j+1},a_{j+1}),\dots,(a_k,a_k))$. It follows that $R(\BH)$ contains $(a_1,\dots,a_{j-1},b,a_{j+1},\dots,a_k)$ and we are done. Then, every element of the form $(x,a)$ or $(a,x)$ is dominated by $(x,b)$ or $(b,x)$ in $\BH^2$, respectively, and hence, $\BH^2$ dismantles to $\BI^2$ where $\BI$ is the structure obtained by dismantling $a$ in $\BH$. It follows by Lemma \ref{le:greedy} that $\BI^2$ dismantles to $\BK$. Iterating this argument we obtain statement (\ref{itemdismantlesrestricted}).
\end{proof}

\begin{lemma}[{[(\ref{itemdismantlesrestricted}) $\Rightarrow$ (\ref{itemdismantles})]}]
\label{le:retraction}
Assume that $\BH$ dismantles to a substructure $\BI$ whose universe contains $J$ and such that $\BI^2$ dismantles to its diagonal. Then, there is a dismantling sequence $\BH^2=\BJ_0,\dots,\BJ_{\ell}=\diagonal(\BI^2)$ such that, for every $0< i\leq \ell$, if $u_i$ is the element folded in $\BJ_{i-1}$ to obtain $\BJ_i$ and $u_i$ is a diagonal element, then $u_i$ is dominated by a (different) diagonal element in $\BJ_{i-1}$.
\end{lemma}

\begin{proof}
Let $a_1,\dots,a_r$ be a sequence of elements to be folded to obtain $\BI$ from $\BH$. We construct a sequence of elements to be folded to obtain $\diagonal(\BI^2)$ from $\BH^2$ as follows: In a first stage, if $a_1$ was dominated by some element $b_1$ in $\BH$, we fold all elements of the form $(x,a_1)$, $(a_1,x)$, or $(a_1,a_1)$ to $(x,b_1)$, $(b_1,x)$, or $(b_1,b_1)$, respectively. In a second stage, one folds (again, in an arbitrary order) all elements of the form $(x,a_2)$, $(a_2,x)$, and $(a_2,a_2)$ that are still left, and continues in the same manner until one obtains $\BI^2$. At this point, one proceeds dismantling $\BI^2$ to its diagonal as originally was done. It is easy to see that the sequence finally obtained satisfies the desired properties.
\end{proof}

We will require the following definition. Let $\BH^2=\BJ_0,\dots,\BJ_{\ell}=\diagonal(\BI^2)$ be the sequence provided by Lemma \ref{le:retraction}. For every $0<i\leq\ell$, let $s_i$ be the fold of $\BJ_{i-1}$ into $\BJ_i$ defined in the natural way (i.e., $s_i$ acts as the identity on $J_i$ and maps $u_i$ to any element that dominates it in $\BJ_{i-1}$) and define, for every $0 < i\leq \ell$,
\begin{equation}
\label{eq:r}
r_i :=s_i \circ \cdots \circ s_1
\end{equation}
and $r_0$ to be the identity. Note that $r_i$ defines a retraction of $\BH^2$ into $\BJ_i$ and that, again by Lemma \ref{le:retraction}, we can assume that $r_i$ maps every diagonal element into a diagonal element. Considering this, we have the following additional lemma, which was inspired by \cite[Lemma 5.2]{1-brightwell}.

\begin{lemma}
\label{le:retraction2}
Under the assumptions of Lemma \ref{le:retraction}, let $\BG$ be a (possibly infinite) $\tau$-structure, $\phi\in \Hom(\BG,\BH^2)$, $X\subseteq G$, and let  $\omega:\BG\rightarrow \BH^2$ be the map defined as
$$
\omega(x) = 
\begin{cases}
(r_{\ell-\dist(x,X)}\circ \phi)(x),	&	\text{ if } \dist(x,X) \leq \ell,	\\
\phi(x),					&	\text{ otherwise},
\end{cases}
$$
where $r_i$ is as in Formula (\ref{eq:r}) and $\dist(x,X)$ is the distance from $x$ to $X$ in $\BG$. Then, $\omega$ is a homomorphism from $\BG$ to $\BH^2$. Furthermore, if $X$ is finite and $\BG$ is locally finite, then $\phi$ and $\omega$ are connected in $C(\BG,\BH^2)$ by a $\diagonal(J^2)$-preserving walk.
\end{lemma}

\begin{proof}
Given $Y \subseteq G$, for every $i=0,\dots,\ell$, let $\omega_{i,Y}$ be the map defined as
$$
\omega_{i,Y}(x) = 
\begin{cases}
(r_{i-\dist(x,Y)}\circ \phi)(x),	&	\text{ if } \dist(x,Y) \leq i,	\\
\phi(x),					&	\text{ otherwise}.
\end{cases}
$$

Note that $\omega=\omega_{\ell,X}$. By induction, we shall prove that $\omega_{i,X}$ is a homomorphism from $\BG$ to $\BH^2$ and, if $X$ is finite and $\BG$ is locally finite, then $\phi$ and $\omega_{i,X}$ are connected in $C(\BG,\BH^2)$ by a $J_i$-preserving walk.

The base case, $i=0$, is trivial. For the inductive case $(i-1\Rightarrow i)$, we can assume that the map $\omega'=\omega_{i-1,X'}$ is a homomorphism in $\Hom(\BG,\BH^2)$, where $X'$ is the set of all elements of $G$ at distance at most $1$ from $X$.

\medskip
\noindent
{\bf Claim.} For every $Z\subseteq X$, the map $\Psi_Z$ that acts as $s_i \circ \omega'$ in $Z$ and as $\omega'$ elsewhere, is a homomorphism from $\BG$ to $\BH^2$. 

\medskip
Indeed, let $ R$ be any relation symbol in $\tau$ and let $\bx=(x_1,\dots,x_k)$ be any tuple in $ R(\BG)$. By inductive hypothesis, $\omega'(\bx)$ belongs to $ R(\BH^2)$. If $\{x_1,\dots,x_k\}\cap Z=\emptyset$, then $\Psi_Z(\bx)=\omega'(\bx)$ and nothing needs to be done. So, assume that $\bx$ contains some element from $Z$. Since $Z\subseteq X$, it follows that $\{x_1,\dots,x_k\}\subseteq X'$. Consequently, $\omega'(\bx) = (r_{i-1} \circ \phi)(\bx)$, which is a tuple of $R(\BJ_{i-1})$. Let $u_i$ be the element that is folded in $\BJ_{i-1}$ to obtain $\BJ_i$ and note that $\Psi_Z(\bx)$ is obtained by replacing, in $\omega'(\bx)$, some (possibly zero) occurrences of $u_i$ by $s_i(u_i)$. Since $s_i(u_i)$ dominates $u_i$ in $\BJ_{i-1}$, it follows that $\Psi_Z(\bx)$ belongs to $ R(\BJ_{i-1})$ (and hence to $ R(\BH^2)$), finishing the proof of the claim.

\medskip
It follows that $\omega=\Psi_X \in\Hom(\BG,\BH^2)$. Furthermore, assume that $X=\{x_1,\dots,x_n\}$ is finite and $\BG$ is locally finite. Then
$$
\Psi_{\emptyset},\Psi_{\{x_1\}},\Psi_{\{x_1,x_2\}},\dots,\Psi_{\{x_1,x_2,\dots,x_n\}}
$$
is a walk in $C(\BG,\BH^2)$ connecting $\omega'=\Psi_{\emptyset}$ and $\omega=\Psi_X$. 

Since $\BG$ is locally finite, it follows that $X'$ is finite. By inductive hypothesis, there is a $J_{i-1}$-preserving walk joining $\phi$ and $\omega'$ in $C(\BG,\BH^2)$. Hence, by concatenating the two walks, it follows that $\phi$ and $\omega$ are also connected in $C(\BG,\BH^2)$. 

It remains to see that the walk thus constructed is $J_i$-preserving. Let $x\in G$ such that $\phi(x)=\omega(x)=(a,b)\in J_i$. Since $(a,b)\in J_i$ it follows from the definition of $\omega'$ that $\omega'(x)=\phi(x)=(a,b)$. Hence it is only necessary to observe that the walk joining $\phi$ and $\omega'$ is $J_i$-preserving (by inductive hypothesis and $J_i\subseteq J_{i-1})$ and that the walk joining $\omega'$ and $\omega$ is also $J_i$-preserving (directly from its construction).
\end{proof}

\begin{lemma}[{[(\ref{itemdismantlesrestricted}) $\Rightarrow$ (\ref{itemC1})]}]
\label{le:connectedness}
If $\BH$ dismantles to a substructure $\BI$ whose universe contains $J$ and such that $\BI^2$ dismantles to its diagonal, then $C(\BG,\BH)$ is $J$-connected for every locally finite $\tau$-structure $\BG$.
\end{lemma}

\begin{proof}
Assume that statement (\ref{itemdismantlesrestricted}) holds. Let $r_0,\dots,r_{\ell}$ be the maps provided by Lemma \ref{le:retraction} (see Equation (\ref{eq:r})), let $\BG$ be any locally finite $\tau$-structure, let $\phi,\psi\in \Hom(\BG,\BH)$, and let $D$ be the set of elements in which $\phi$ and $\psi$ disagree, which we can assume is finite. Then, the map $\Phi: G \rightarrow H^2$ with $x \mapsto \Phi(x)=(\phi(x),\psi(x))$ for $x \in G$, defines a homomorphism from $\BG$ to $\BH^2$. Since $D$ is finite, it follows from Lemma \ref{le:retraction2} that there is $\diagonal(J^2)$-preserving walk $\Psi_1,\dots,\Psi_\seq$ in $C(\BG,\BH^2)$ that joins $\Phi$ and the map $\omega:G\rightarrow H^2$ sending $x\in G$ to $(r_{\ell-j} \circ \Phi)(x)$, where $j$ is the minimum between $\ell$ and the distance from $x$ to $D$. It is not difficult to see that $\omega(x)$ is a diagonal element for every $x\in G$. Indeed, if $x\not\in D$, it follows from the fact (seen just right after Lemma \ref{le:retraction}) that every retraction $r_{\ell-j}$ maps diagonal elements into diagonal elements and, if $x \in D$, it follows from the fact that $r_{\ell}(\BH^2)=\diagonal(\BI^2)$. Hence, $\pi_1 \circ \Psi_\seq = \pi_2 \circ \Psi_\seq$. Consequently, $\pi_1 \circ \Psi_1,\pi_1 \circ \Psi_2,\dots, \pi_1 \circ \Psi_\seq = \pi_2 \circ \Psi_\seq, \pi_2 \circ \Psi_{\seq-1},\dots,\pi_2 \circ \Psi_1$ defines a walk that joins $\pi_1 \circ \Psi_1=\phi$ and $\pi_2 \circ \Psi_1=\psi$. It remains to see that the walk thus defined is $J$-preserving. To do so, we just have to notice that if $x$ is an element in $\BG$ such that $\phi(x)=\psi(x)\in J$, then we have by construction that $\Psi_1(x)=(x,x)$ and $\Psi_\seq(x)=(x,x)$, and then use the fact that $\Psi_1,\dots,\Psi_\seq$ is $\diagonal(J^2)$-preserving.
\end{proof}

\begin{lemma}[{[(\ref{itemdismantlesrestricted}) $\Rightarrow$ (\ref{itemgraph})]}]
\label{le:interpolate}
If $\BH$ dismantles to a substructure $\BI$ whose universe contains $J$ and such that $\BI^2$ dismantles to its diagonal, then there exists $\gap \geq 0$ such that $\Hom(\BG,\BH)$ is strongly $J$-irreducible with gap $\gap$ for every $\tau$-structure $\BG$.
\end{lemma}

\begin{proof} 
Assume that statement (\ref{itemdismantlesrestricted}) holds. Let $r_0,\dots,r_{\ell}$ be the maps provided by Lemma \ref{le:retraction} (see Equation (\ref{eq:r})) and define $\gap$ to be $2\ell$. Let $\BG$ be a $\tau$-structure, $V, W \subseteq G$ with $\dist(V,W) \geq g$, and  $\phi, \psi \in \Hom(\BG,\BH)$. Then, the map $x\mapsto (\phi(x),\psi(x))$ is a homomorphism from $\BG$ to $\BH^2$. Let $\omega$ be the homomorphism from $\BG$ to $\BH^2$ provided by Lemma \ref{le:retraction2} with map $x\mapsto (\phi(x),\psi(x))$ and where $X\subseteq G$ is the set of all elements at distance at least $\ell$ from $V\cup W$. It follows by construction that $\omega(x)=(\phi(x),\psi(x))$ for every $x\in V\cup W$ and that $\omega(X)\subseteq\diagonal(I^2)$.

Define $h(x)$ to be $\pi_1(\omega(x))$ if $x$ is at distance less than $\ell$ to $V$ and $\pi_2(\omega(x))$ otherwise. We shall prove that $h$ defines a homomorphism from $\BG$ to $\BH$. Let $R$ be any relation symbol in $\tau$ of arity $\arity$, and let $\bx=(x_1,\dots,x_\arity)$ be any tuple in $R(\BG)$. The fact that $h$ is a homomorphism follows directly from the following claim: $h(\bx) = \pi_1(\omega(\bx))$ if the minimum distance $j$ from any element in the tuple to $V$ is at most $\ell-1$ and $h(\bx) = \pi_2(\omega(\bx))$ otherwise. To prove the claim, notice that, since all the elements in $\bx$ are at distance $j$ or $j+1$ from $V$, we only need to consider the case when $j=\ell-1$. Let $x_i$ be any element in $\bx$. If the distance of $x_i$ to $V$ is $\ell-1$, then $h(x_i)=\pi_1(\omega(x_i))$ by definition. Assume, otherwise, that the distance of $x_i$ to $V$ is $\ell$. Since the distance from $V$ to $W$ is at least $\gap=2\ell$, then $x_i$ is at distance at least $\ell$ from $W$, and hence $x_i\in X$. Hence, $\omega(x_i)$ is a diagonal element, and hence $\pi_2(\omega(x_i))=\pi_1(\omega(x_i))$.

Furthermore, since for every $x\in V\cup W$, $\omega(x)=(\phi(x),\psi(x))$, it follows that $h$ agrees with $\phi$ on $V$ and $\psi$ on $W$. 

Now, let $x\in G$ be any element such that $\phi(x)=\psi(x)\in J$. Note
that, by construction, $\omega$ is $\diagonal(I^2)$-preserving. 
Since $J\subseteq I$, it follows that $\omega(x)=(\phi(x),\psi(x))$. Since $h(x)$ is either the first or second projection of $\omega(x)$, it follows that $h(x)=\phi(x)=\psi(x)$.
\end{proof}

Recall that the definition of the label map  $\rho_{\BH}: T_\BH \to H$ that sends every walk $w$ in $T_{\BH}$ to its ending point (see Section \ref{forest}).

\begin{lemma}[{[(\ref{itemtree}) $\Rightarrow$ (\ref{itemdismantlesrestricted})]}]
\label{le:free-tree} 
If there exists $\gap \geq 0$ such that $\Hom(\BT_{\BH^2},\BH)$ is $(\{x\},W)$-mixing with respect to $J$ with gap $\gap$ for all $x \in T_{\BH^2}$ and $W \subseteq T_\BH^2$, then $\BH$ dismantles to a substructure $\BI$ whose universe contains $J$ and such that $\BI^2$ dismantles to its diagonal.
\end{lemma}

\begin{proof}
For the sake of contradiction, let's suppose that statement (\ref{itemtree}) holds, but statement (\ref{itemdismantlesrestricted}) does not. Let $\BI$ be any $J$-non-foldable relational structure obtained by dismantling $\BH$, let $\BK$ be the symmetric $\diagonal(J^2)$-non-foldable relational structure given by Lemma \ref{le:symmetric}, let $r$ be a retraction of $\BH^2$ onto $\BK$ defined in the natural way, and let $\gap$ be any value given by statement $(\ref{itemtree})$.

Since statement (\ref{itemdismantlesrestricted}) does not hold, it follows that $K$ contains a non-diagonal element $a=(a_1,a_2)$. Let $w_0$ be the (unique) walk of length $0$ starting at $a$, let $V=\{w_0\}$, and let $W_\gap$ be the set containing all walks in $\BT_{\BH^2}$  of length at least $\gap$. Let $\phi$ and $\psi$ be the homomorphisms from $\BT_{\BH^2}$ to $\BH$ defined as $\phi := \pi_1 \circ \rho_{\BH^2}$ and $\psi := \pi_2 \circ \rho_{\BH^2}$. From statement (\ref{itemtree}), it follows that there exist homomorphisms $h_1,h_2: \BT_{\BH^2}\rightarrow\BH$ such that $h_1$ agrees with $\psi$ on $V$ and with $\phi$ on $W_\gap$, $h_2$ agrees with $\phi$ on $V$ and with $\psi$ on $W_\gap$, and $h_1(w) = h_2(w) = \phi(w) = \psi(w)$ for all $w \in \gamma^{-1}(\Delta(J^2))$.

Since $\BT_{\BK}$ is a substructure of $\BT_{\BH^2}$, it follows that the mapping $\rho':T_{\BK}\rightarrow K$ defined as $\rho'(w)=r(h_1(w),h_2(w))$ is a homomorphism from $T_{\BK}$ to $\BK$. By construction, $\rho'$ agrees with $\rho_{\BK}$ in $W_\gap$ and every $w\in T_{\BK}$ with $\rho_{\BK}(w)\in\diagonal(J^2)$. Hence, by Lemma~\ref{le:universaltree}, $\rho'$ must be identical to $\rho_{\BK}$, implying, in particular, that $\rho_{\BK}$ and $\rho'$ agree in $w_0$. However, 
$$
\rho'(w_0)=r(h_1(w_0),h_2(w_0))=r(a_2,a_1)=(a_2,a_1),
$$
where the last equality follows from the fact that $\BK$ is symmetric. We obtain a contradiction, since $\rho_{\BK}(w_0)=a=(a_1,a_2)$ and $a$ is a non-diagonal element.
\end{proof}

\section{Gibbs measures and applications}
\label{section5}

\subsection{Basic definitions}

Given a finite $\tau$-structure $\BH$ with universe $H$, a {\bf weight function} for $\BH$ is a map $\lambda: H \to \mathbb{R}^+$.

Let $\BG$ be a locally finite $\tau$-structure. If $V \subseteq G$ is a finite set and $\phi \in \Hom(\BG,\BH)$, we define $\mathbb{P}_{V,\phi}$ to be the probability measure on $\Hom(\BG,\BH)$ given by
$$
\mathbb{P}_{V,\phi}(\{\psi\}) := 
\begin{cases}
Z_{V,\phi}(\lambda)^{-1} \prod_{x \in V} \lambda(\psi(x))	&	\text{if } \left.\psi\right|_V \cup \left.\phi\right|_{G \setminus V}  \in \Hom(\BG,\BH),	\\
0											&	\text{otherwise},
\end{cases}
$$
for $\psi \in \Hom(\BG,\BH)$, where $\left.\psi\right|_V \cup \left.\phi\right|_{G \setminus V}$ is the map that coincides with $\psi$ in $V$ and with $\phi$ in $G \setminus V$, and $Z_{V,\phi}(\lambda)$ is a normalization constant---the \emph{partition function}---defined as
$$
Z_{V,\phi}(\lambda) := \sum_{\substack{\psi \in \Hom(\BG,\BH)\\ \left.\psi\right|_V \cup \left.\phi\right|_{G \setminus V}  \in \Hom(\BG,\BH)}} \prod_{x \in V} \lambda(\psi(x)).
$$

We will call the collection of probability measures $\{\mathbb{P}_{V,\phi}\}$, the {\bf Gibbs $(\BG,\BH,\lambda)$-specification}. The {\bf boundary} of a set $V \subseteq G$, denoted by $\partial V$, is defined as the set of elements in $G$ at distance exactly $1$ from $V$. Notice that $\mathbb{P}_{V,\phi}$ depends exclusively on $\left.\phi\right|_{\partial V}$. Now, consider events of the form
$$
A(\phi,V) = \left\{\psi \in \Hom(\BG,\BH): \left.\psi\right|_V = \left.\phi\right|_V\right\}.
$$

Next, consider the $\sigma$-algebra $\mathcal{F}$ generated by all events of the form $A(\phi,V)$ for $V$ finite, and define $\mathcal{M}(\BG,\BH)$ to be the set of probability measures on $(\Hom(\BG,\BH),\mathcal{F})$.

A measure $\mu \in \mathcal{M}(\BG,\BH)$ is a {\bf Gibbs measure} for the Gibbs $(\BG,\BH,\lambda)$-specification if for any finite $V \subseteq G$ and for all $\psi \in \Hom(\BG,\BH)$,
$$
\mu\left(A(\psi,V) \middle\vert A(\phi, G \setminus V) \right) = \mathbb{P}_{V,\phi}\left(\{\psi\}\right) \text{ for $\mu$-a.e. $\phi \in \Hom(\BG,\BH)$}.
$$

In other words, the probability distribution of a random $\psi$ inside a finite $V$ conditioned on its values outside $V$ to coincide with those of $\phi$, depends only on the values of $\left.\psi\right|_V$ and on the boundary, $\left.\phi\right|_{\partial V}$. Furthermore, the conditional distribution is the same as for $\mathbb{P}_{V,\phi}$ (see also \cite[Definition 2.1]{1-brightwell}).

If $\Hom(\BG,\BH) \neq \emptyset$, then there always exists at least one Gibbs measure \cite[Chapter 4]{1-georgii}. A fundamental question in statistical physics is whether there exists a unique Gibbs measure or multiple for a given Gibbs $(\BG,\BH,\lambda)$-specification.

\subsection{Non-uniqueness and spatial mixing properties}

In \cite{1-brightwell}, it is shown that if $\BH$ is a graph and it is dismantlable (or equivalently, by Lemma \ref{lem:graphdism}, its square dismantles to a subgraph of its diagonal), then, for any graph $\BG$ of bounded degree (and therefore, locally finite), there exists some $\lambda$ such that there is a unique Gibbs measure \cite[Theorem 7.2]{1-brightwell}. Conversely, in \cite{1-brightwell} it is also proved that if $\BH$ is a non-dismantlable graph, then there exists $\BG$ (of bounded degree) such that for any $\lambda$ there exists multiple Gibbs measures \cite[Theorem 8.2]{1-brightwell}.

Here, following a similar path, we show that when extending this question to arbitrary relational structures, the first implication does not remain true in general, but the second still holds. More exactly,

\begin{proposition}
\label{lem:counterex}
There exists a finite $\tau$-structure $\BH$ such that $\BH^2$ dismantles to a substructure of its diagonal and a $\tau$-structure $\BG$ of bounded degree such that for any $\lambda$ there exists multiple Gibbs measures for the Gibbs $(\BG,\BH,\lambda)$-specification. Moreover, $\BH$ can be chosen so that $\BH^2$ dismantles to its full diagonal $\diagonal(\BH^2)$.
\end{proposition}

\begin{proposition}
\label{lem:frozengibbs}
Let $\BH$ be a finite $\tau$-structure. If $\BH^2$ does not dismantle to a substructure of $\diagonal(\BH^2)$, then there exists a $\tau$-structure $\BG$ of bounded degree such that for any $\lambda$ there exists multiple Gibbs measures for the Gibbs $(\BG,\BH,\lambda)$-specification.
\end{proposition}

Before proving these two results, we introduce and explore some \emph{spatial mixing} properties in this same context.

\begin{defi}
\label{def:Jsm}
Given $J \subseteq H$, we say that a Gibbs $(\BG,\BH,\lambda)$-specification satisfies {\bf spatial $J$-mixing ($J$-SM)} if there exists constants $C, \alpha > 0$ such that for all $\phi_1,\phi_2 \in \Hom(\BG,\BH)$, for all finite $V \subseteq G$, and for all $x \in V$ and $a \in H$,
\begin{equation}
\label{eq:spatial}
\left|\mathbb{P}_{V,\phi_1}(\{\psi(x) = a\}) - \mathbb{P}_{V,\phi_2}(\{\psi(x) = a\})\right|	 \leq C \cdot \exp(-\alpha \cdot \dist(x,D^J_V(\phi_1,\phi_2))),
\end{equation}
where 
$$
D^J_V(\phi_1,\phi_2) = \{x \in \partial V: (\phi_1(x), \phi_2(x)) \in H^2 \setminus \diagonal(J^2)\}
$$
and $\{\psi(x) = a\}$ refers to the event that a random $\psi$ takes the value $a$ at $x$.
\end{defi}

The definition of $J$-SM unifies and interpolates two well-known properties. If $J = \emptyset$, then $D^\emptyset_V(\phi_1,\phi_2) = \partial V$ and Equation (\ref{eq:spatial}) corresponds to the definition of {\bf weak spatial mixing (WSM)}, i.e., $\emptyset$-SM. On the other hand, if $J = H$, then $D^H_V(\phi_1,\phi_2) = \{x \in \partial V: \phi_1(x) \neq \phi_2(x)\}$ and Equation (\ref{eq:spatial}) corresponds to the definition of {\bf strong spatial mixing (SSM)}, i.e., $H$-SM.

In general, spatial mixing properties are forms of \emph{correlation decay} that have been of interest because of their many applications. On the one hand, WSM is related with uniqueness of Gibbs measures and the absence of phase transitions \cite{1-dyer}. On the other hand, SSM is a strengthening of WSM and it is related to the absence of \emph{boundary phase transitions} \cite{1-martinelli} and has connections with the existence of FPTAS for \#P-hard counting problems \cite{bandyopadhyay2008,1-weitz}, mixing time of Glauber dynamics \cite{1-dyer}, and efficient approximation algorithms for thermodynamic quantities \cite{gamarnik2009,1-briceno}.

In \cite{2-briceno}, there were explored sufficient and necessary conditions for a graph $\BH$ to have, for any graph $\BG$ of bounded degree, the existence of a weight function $\lambda$ such that the Gibbs $(\BG,\BH,\lambda)$-specification satisfies WSM and SSM. In particular, it was proved that dismantlability was equivalent to the existence of Gibbs $(\BG,\BH,\lambda)$-specifications satisfying WSM for any graph $\BG$ of bounded degree, and therefore uniqueness, since WSM implies it. In addition, it was observed that a direct consequence is that a necessary condition for SSM to hold is that $\BH$ is dismantlable, because SSM implies WSM. However, it was also shown that it is not a sufficient condition. Here, we strengthen this necessary condition and extend it to the realm of relational structures.

\begin{proposition}
\label{lem:nossm}
If $\BH^2$ does not dismantle to a substructure of $\diagonal(\BH^2)$ whose universe contains $\diagonal(J^2)$, then there exists a $\tau$-structure $\BG$ of bounded degree such that the Gibbs $(\BG,\BH,\lambda)$-specification does not satisfy $J$-SM for any $\lambda$.
\end{proposition}

Two direct corollaries of this fact are the following.

\begin{corollary}
\label{cor:nowsm}
If $\BH^2$ does not dismantle to some substructure of the diagonal $\diagonal(\BH^2)$, then there exists a $\tau$-structure $\BG$ of bounded degree such that the Gibbs $(\BG,\BH,\lambda)$-specification does not satisfy WSM for any $\lambda$.
\end{corollary}

\begin{corollary}
\label{cor:nossm}
If $\BH^2$ does not dismantle to the full diagonal $\diagonal(\BH^2)$, then there exists a $\tau$-structure $\BG$ of bounded degree such that the Gibbs $(\BG,\BH,\lambda)$-specification does not satisfy SSM for any $\lambda$.
\end{corollary}

\subsection{Proofs of the propositions}

Before proving the propositions, we will need to state and prove the following lemma.

\begin{lemma}
\label{lemmafrozen}
Let $\BH$ be a finite $\tau$-structure with universe $H$ and let $J \subseteq H$. Suppose that $\BH^2$ does not dismantle to a substructure of $\diagonal(\BH^2)$ whose universe contains $\diagonal(J^2)$. Then, there exists an infinite $\tau$-structure of bounded degree $\BG$, $x_0 \in G$, and $\rho \in \Hom(\BG,\BH^2)$ such that, for $\phi_i=\pi_i \circ \rho$ ($i=1,2$) and any cofinite set $U \subseteq G \setminus \{x_0\}$ with $\rho^{-1}(\Delta(J^2)) \subseteq U$, we have that
\begin{equation}
\label{eq:X} X_1\cap X_2=\emptyset,
\end{equation} 
for $X_i := \{ \psi(x_0) : \psi\in \Hom(\BT_\BK,\BH), \left.\psi\right|_U = \left.\phi_i\right|_U\}$ and $i=1,2$.
\end{lemma}

\begin{proof}
Suppose that $\BH^2$ does not dismantle to a substructure $\BK$ with universe $K$ satisfying $\diagonal(J^2) \subseteq K \subseteq\diagonal(H^2)$. Let $\BI$ be any $J$-non-foldable relational structure obtained by dismantling $\BH$ and let $\BK$ be the $\diagonal(J^2)$-non-foldable relational structure given by Lemma~\ref{le:symmetric}. It follows easily (for example, see the proof of Lemma \ref{le:retraction}) that $\BH^2$ dismantles to $\BK$ and that there is a retraction $r$ from $\BH^2$ to $\BK$ such that the image of every diagonal element is also a diagonal element (that is, such that $r(\diagonal(H^2))\subseteq\diagonal(H^2)$).

By our assumption, $\BK$ contains some non-diagonal element $(a_1,a_2)$. Let $w_0\in T_{\BK}$ be the (unique) walk of length $0$ starting in $(a_1,a_2)$, and let $U\subseteq T_{\BK}$ be any cofinite set such that $w_0\not\in U$ and $\rho_{\BK}^{-1}(\Delta(J^2))\subseteq U$.

Note that $\Hom(\BT_\BK,\BH)$ contains mappings $\phi_i=\pi_i \circ \rho_\BK$ for $i=1,2$. Also, let 
$$X_i=\{ \psi(w_0) : \psi\in \Hom(\BT_\BK,\BH), \left.\psi\right|_U = \left.\phi_i\right|_U\}.$$

We claim that 
\begin{equation}
\label{eq:X} X_1\cap X_2=\emptyset.
\end{equation} 

Let us prove it by contradiction. Assume that $\psi_1(w_0)=\psi_2(w_0)$, where $\psi_i \in \Hom(\BT_\BK,\BH)$ and $\left.\psi_i\right|_U = \left.\phi_i\right|_U$ for $i=1,2$. Note that $\Psi(x_0)=r\circ (\psi_1(x_0),\psi_2(x_0))$ defines a homomorphism from $\BT_\BK$ to $\BK$ that agrees with $\rho_{\BK}$ in $U$. Since $\BK$ is $\diagonal(J^2)$-non-foldable and $\rho_{\BK}^{-1}(\diagonal(J^2))\subseteq U$, it follows by Lemma \ref{le:universaltree} that $\Psi(w_0)=(a_1,a_2)$. However, this is a contradiction since $(\psi_1(w_0),\psi_2(w_0))$ is a diagonal element and $r$ sends diagonal elements to diagonal elements. Finally, by identifying $\BG$ with $\BT_\BK$, $\rho$ with $\rho_{\BK}$, and $a$ with $w_0$, we conclude.
\end{proof}

\begin{proof}[Proof of Proposition \ref{lem:frozengibbs}]
Let $\BG$, $x_0 \in G$, and $\rho$ be as in Lemma \ref{lemmafrozen} for $J = \emptyset$. Let $\lambda$ be any weight function for $\BH$. Construct a Gibbs measure $\mu_1$ by taking \emph{weak limits} using $\phi_1 = \pi_1 \circ \rho$, this is to say, we consider the sequence of measures $\{\mathbb{P}_{V_n,\phi_1}\}_n$, where $V_n$ is an increasing sequence (in the sense of inclusion) of finite sets containing $x_0$ eventually exhausting $\BG$. Such a sequence must have a subsequence $\{\mathbb{P}_{V_{n_k},\pi_1 \circ \rho}\}_k$ \emph{weakly converging} to a Gibbs measure $\mu_1$ (i.e., $\lim_{k \to \infty} \mathbb{P}_{V_{n_k},\pi_1 \circ \rho}(A(\phi,V)) = \mu_1(A(\phi,V))$ for all $\phi$ and finite $V$). This is a standard argument used for constructing Gibbs measures \cite[Chapter 4]{1-georgii}. Similarly, the sequence $\{\mathbb{P}_{V_{n_k},\pi_2 \circ \rho}\}_k$ must contain a subsequence converging to a Gibbs measure $\mu_2$. 

For every $a \in H$, we have that
$$\lim_{k \to \infty} \mathbb{P}_{V_{n_k},\phi_1}(\{\psi(x_0) = a\}) = \mu_1(\{\psi(x_0) = a\}).$$

Hence, choose any $a\in H$ such that $\mu_1(\{\psi(x_0) = a\})>0$ (that it has to exist, since $\mu_1$ is a probability measure). It follows from Equation (\ref{eq:X}), by setting $U = G \setminus V_{n_k}$, that 
$$\lim_{k \to \infty} \mathbb{P}_{V_{n_k},\phi_2}(\{\psi(x_0) = a\}) = 0.$$

Consequently, $\mu_2(\{\psi(x_0) = a\})=0$ and hence, $\mu_1\neq\mu_2$. 
\end{proof}

\begin{proof}[Proof of Proposition \ref{lem:nossm}]
This is a direct consequence of the results proved in Proposition \ref{lem:frozengibbs}. Let us fix some weight function $\lambda$ for $\BH$. Assume $\BH^2$ does not dismantle to a substructure of $\diagonal(\BH^2)$ whose universe contains $\diagonal(J^2)$ and let $\BG$, $x_0$, and $\rho$ as in Lemma \ref{lemmafrozen}. Define $\phi_i = \pi_i \circ \rho$ for $i=1,2$. Pick $m$ large enough and let $V\subseteq G$ be the set of all elements $x \in G$ that are at distance less than $m$ from $x_0$ and such that $\rho(x)\not\in\diagonal(J^2)$. 

Since $\mathbb{P}_{V,\phi_1}$ is a probability measure, there exists some element $a \in H$ such that 
$$
\mathbb{P}_{V,\phi_1}(\{\psi(x_0) = a\}) \geq \frac{1}{|H|}.
$$

It follows from Equation (\ref{eq:X}) in Proposition \ref{lem:frozengibbs} (by setting $U = G\setminus V$), that
$$
\mathbb{P}_{V,\phi_2}(\{\psi(x_0) = a\})=0.
$$

Consequently, we have that $\mathbb{P}_{V,\phi_1}(\{\psi(x_0) = a\})-\mathbb{P}_{V,\phi_2}(\{\psi(x_0) = a\}) \geq \frac{1}{|H|}$.
Note that $\dist(x_0,D_V^J(\phi_1,\phi_2))\geq m$. Indeed, it follows from the definition of $V$ that if the distance of $x \in \partial V$ to $x_0$ is less than $m$, then $\rho(x)\in\diagonal(J^2)$, which implies that $(\phi_1(x), \phi_2(x)) \notin H^2 \setminus \diagonal(J^2)$. It follows that for any constants $C,\alpha>0$, the quantity $C \cdot \exp(-\alpha \cdot \dist(x_0,D_U^J(\phi_1,\phi_2)))$ from the definition of $J$-SM can be made arbitrarily small by choosing $m$ large enough. Therefore, $J$-SM cannot hold.
\end{proof}

\begin{proof}[Proof of Proposition \ref{lem:counterex}]
Let $\tau = \{R_1,R_2,R_3\}$ be a signature with $R_i$ a $2$-ary relation for $i=1,2,3$. Consider the $\tau$-structure $\BH$ with universe $H = \{0,1,2\}$ and
\begin{itemize}
\item $R_1(\BH) = \{(0,0),(0,1),(1,0)\}$,
\item $R_2(\BH) = \{(1,1),(1,2),(2,1)\}$, and
\item $R_3(\BH) = \{(2,2),(2,0),(0,2)\}$.
\end{itemize}

It can be checked that $\BH^2$ dismantles to $\diagonal(\BH^2)$. Indeed, it suffices to fold $(0,1)$ and $(1,0)$ to $(0,0)$, $(1,2)$ and $(2,1)$ to $(1,1)$, and $(0,2)$ and $(2,0)$ to $(2,2)$. This $\tau$-structure is intimately related to the so-called \emph{hardcore model}, a system well studied in combinatorics and statistical physics \cite{1-galvin,1-weitz} consisting of a Gibbs $(\BG,\BH,\lambda_{\mathrm{HC}})$-specification for an arbitrary graph $\BG$, the graph $\BH$ with universe $\{a,b\}$ (where we think that $a$ is a \emph{particle} and $b$ is a \emph{non-particle}) and edge relation $E = \{(a,b),(b,b)\}$, and a weight function $\lambda_{\mathrm{HC}}$ such that $\lambda_{\mathrm{HC}}(b) = 1$ and $\lambda_{\mathrm{HC}}(a) > 0$, usually called \emph{activity}.

Now, consider the $\tau$-structure $\BG$ consisting of $3$ copies of the \emph{$\Delta$-regular tree} for an arbitrary $\Delta \geq 6$, so that the graph adjacency relation in the $i$th copy is given by $R_i$, for $i=1,2,3$. Let $\lambda:\{0,1,2\} \to \mathbb{R}^+$ be an arbitrary weight function. Then, we can think that we have $3$ copies of the hardcore model (with values $\{0,1\}$, $\{1,2\}$, and $\{2,0\}$, respectively) on a $\Delta$-regular tree with activities $\lambda(1)/\lambda(0)$, $\lambda(2)/\lambda(1)$, and $\lambda(0)/\lambda(2)$, respectively (i.e. the ratios between the weight of particle versus non-particle). By symmetry, w.l.o.g., suppose that $\lambda(1)/\lambda(0) \geq 1$. The \emph{critical activity} for the hardcore model in a $\Delta$-regular tree is given by the formula $\frac{(\Delta-1)^\Delta}{(\Delta-2)^\Delta}$ \cite{kelly1985,1-weitz}, i.e., if the activity is below this threshold there exists a unique Gibbs measure (\emph{subcritical regime}), and if it is below, there exist multiple ones (\emph{supercritical regime}). Since $\frac{(\Delta-1)^\Delta}{(\Delta-2)^\Delta} < 1 \leq \lambda(1)/\lambda(0)$ for $\Delta \geq 6$, then the first copy corresponds to a hardcore model  in a $\Delta$-regular tree in the supercritical regime, inducing multiple Gibbs measure on $\Hom(\BG,\BH)$.
\end{proof}

\section{Finite duality revisited}
\label{section6}

Throughout this section all relational structures are assumed to be finite. We say that a $\tau$-structure $\BH$ is a {\bf core} if every homomorphism from $\BH$ to $\BH$ is one-to-one. An {\bf obstruction} to $\BH$ is a $\tau$-structure $\mathbb{O}$ that admits no homomorphism to $\BH$; the obstruction $\mathbb{O}$ is {\bf critical} if every \emph{proper} substructure (i.e., any substructure different from $\mathbb{O}$ itself) admits a homomorphism to $\BH$. A relational structure $\BH$ is said to have {\bf finite duality} if it has only finitely many critical obstructions $\{\mathbb{O}_1,\dots,\mathbb{O}_m\}$. This implies that, for every $\tau$-structure $\BG$,
$$
\Hom(\mathbb{O}_i,\BG) \neq \emptyset \text{ for some } 1 \leq i \leq m \iff \Hom(\BG,\BH) = \emptyset.
$$

We say that a $\tau$-structure $\BH$ {\bf contains all  constants} if for every $a\in H$ there exists $R_a\in\tau$ such that $R_a(\BH)=\{a\}$. Note that every such relational structure is a core. 

The main result in \cite{MR2357493} states that a core relational structure $\BH$ has finite duality if and only if $\BH^2$ dismantles to its diagonal. In this section we shall see how this result follows from Theorem~\ref{the:main}. In addition, we shall show that, when $\BH$ contains all constants, having finite duality is equivalent to having finitely many critical $\tau$-tree obstructions, which was not previously known.

\begin{theorem}
\label{the:core}
Let $\BH$ be a finite $\tau$-structure which is a core. Then, the following are equivalent:
\begin{enumerate}
\item[(\mylabel{A1c}{A1c})] $\BH^2$ dismantles to its full diagonal;
\item[(\mylabel{D1c}{D1c})] $\BH$ has finitely many critical obstructions.
\end{enumerate}

Furthermore, if $\BH$ contains all the constants, then the following statement is also equivalent:
\begin{enumerate}[resume]
\item[(\mylabel{D2c}{D2c})] $\BH$ has finitely many critical $\tau$-tree obstructions.
\end{enumerate}
\end{theorem}

\begin{proof}
As mentioned earlier, the equivalence $(\ref{A1c})\Leftrightarrow(\ref{D1c})$ was shown in \cite{MR2357493}. Here we provide an alternative proof.

\medskip
\noindent
$(\ref{A1c})\Rightarrow(\ref{D1c})$. Assume that statement (\ref{A1c}) holds. It follows that statement Theorem \ref{the:main}(\ref{itemgraph}) holds for any $J\subseteq H$ (although we note that in this proof it suffices the case $J=\emptyset$). We claim that the diameter of the critical obstructions of $\BH$ is bounded, where the diameter of a relational structure is the maximum distance between any pair of its elements. It follows easily (see \cite[Lemma 2.4]{MR2357493}) that our claim implies statement (\ref{D1c}).

To prove the claim, assume towards a contradiction that there exists a critical obstruction $\BG$ containing two elements $x$ and $y$ at distance at least $\gap+2$  (where $\gap$ is the gap given by statement Theorem \ref{the:main}(\ref{itemgraph})). Let $R\in\tau$ and let $\bx=(x_1,\dots,x_k)\in R(\BG)$ in which $x$ occurs. Since $\BG$ is critical, it follows that the substructure obtained by removing $\bx$ from $R(\BG)$ has a homomorphism $\phi$ to $\BH$. Similarly, the substructure obtained by removing from some relation $S(\BG)$ a tuple $\by=(y_1,\dots,y_{k'})$ where $y$ occurs, has a homomorphism $\psi$ to $\BH$. 

Hence, both $\phi$ and $\psi$ define homomorphisms from the substructure $\BK$ obtained from $\BG$ by removing both $\ba$ from $R(\BG)$ and $\bb$ from $S(\BG)$. Note that the distance, in $\BK$, from $V=\{x_1,\dots,x_k\}$ to $W=\{y_1,\dots,y_{k'}\}$ is at least $\gap$. It follows from Theorem \ref{the:main}(\ref{itemgraph}) that there is a homomorphism $\gamma$ from $\BK$ to $\BH$ that agrees with $\phi$ on $V$ and with $\psi$ with $W$. Consequently, $\gamma$ defines a homomorphism from $\BG$ to $\BH$, a contradiction.

\medskip
\noindent
$(\ref{D1c})\Rightarrow(\ref{D2c})$. It is immediate, since critical tree obstructions are critical obstructions.

\medskip
\noindent
$(\ref{D2c})\Rightarrow(\ref{A1c})$. Let us prove the contrapositive. Assume that $\BH^2$ does not dismantle to the diagonal or, equivalently, that $\BH$ does not satisfy statement Theorem \ref{the:main}(\ref{itemdismantlesrestricted}) when $J=H$. It follows that $\BH$ does not satisfy statement Theorem \ref{the:main}(\ref{itemtree}) either. Let $\gap$ be an arbitrary gap. Since the statement Theorem \ref{the:main}(\ref{itemtree}) fails, it follows by standard compactness arguments that there is a finite substructure $\BG$ of $\BT_{\BH^2}$ and $V, W \subseteq G \subseteq T_{\BH^2}$ with $\dist(V,W)\geq \gap$ such that $\Hom(\BG,\BH)$ is not $(V,W)$-mixing with respect to $H$. That is, there exists mappings $\phi,\psi \in \Hom(\BG,\BH)$ such that there is no mapping in $\Hom(\BG,\BH)$ that agrees with $\phi$ on $V$, with $\psi$ on $W$, and with both on every element $x\in G$ such that $\phi(x)=\psi(x)$.

Now, let $\BK$ be the $\tau$-structure obtained from $\BG$ by coloring every element $x\in V$ according to $\phi$, every element in $W$ according to $\psi$, and every element $x$ such that $\phi(x)=\psi(x)\in J$ according to either $\phi$ or $\psi$. Since there is no homomorphism from $\BK$ to $\BH$ and, consequently, there is a substructure $\BI$ of $\BK$ that is a critical obstruction of $\BH$. Since $\BI$ is critical, then it is connected. Consequently, since $\BG$ is a substructure of $\BT_{\BH^2}$ and $\BT_{\BH^2}$ does not contain cycles, it follows that $\BI$ is a $\tau$-tree. Clearly, $V\cap I$ is nonempty since otherwise the mapping $x\mapsto\psi(x)$ would define a homomorphism from $\BI$ to $\BH$. Similarly, $W\cap I\neq\emptyset$. Since $\BI$ is connected and the distance in $\BG$ (and hence in $\BI$) from $V$ to $W$ is at least $\gap$, it follows that $|I|\geq\gap+1$. Since $\gap$ is arbitrary, we have completed the proof.

\medskip
\noindent
$(\ref{D1c})\Rightarrow(\ref{A1c})$.
Assume that $\BH$ satisfies statement (\ref{D1c}). Let $\BH_c$ obtained by endowing $\BH$ with all constants. Formally, if $\tau$ is the signature of $\BH$, then $\tau_c$ is the new signature containing a new relation symbol $R_a$ for every $a\in H$, and $\BH_c$ is the $\tau_c$-structure obtained from $\BH$ by setting $R_a(\BH_c)=\{a\}$ for every $a\in H$. We shall show that $\BH_c$ has also finite duality. Consequently, $\BH_c$ satisfies statement (\ref{D2c}) and, hence, $\BH_c^2$ dismantles to the diagonal, implying that $\BH^2$ dismantles to the diagonal as well. 

The proof of this claim is fairly standard. Let $\BK_c$ be a minimal critical obstruction of $\BH_c$ and let $\BG$ be the $\tau$-structure constructed in the following way. In a first stage, consider the disjoint union of $\BH$ and $\BK$, where $\BK$ is the $\tau$-structure obtained by removing all constants from $\BK_c$ (that is, $\BK$ is the $\tau$-structure obtained from $\BK_c$ by removing all relations in $\tau_c\setminus \tau$). 
In a second stage, we glue some elements from $\BH$ and $\BK$. In particular we glue every element $a$ in $\BH$ to every element $b \in R_a(\BK_c)$. We shall show
that $\BG$ is not homomorphic to $\BH$.  Assume towards a contradiction that there is $\phi$ that defines a homomorphism from $\BG$ to $\BH$. Clearly, the restriction of $\phi$ to $H$, that we shall denote $\left.\phi\right\vert_H$, defines a homomorphism of $\BH$ that must be one-to-one since $\BH$ is a core. Since $\left.\phi\right\vert_H$ is one-to-one it follows that $\left.\phi\right\vert_H^{-1}$ is also a homomorphism from $\BH$ to $\BH$ and, hence, $\left.\phi\right\vert_H^{-1} \circ \phi$ defines a homomorphism from $\BG$ to $\BH$ that acts as the identity on $\BH$. It follows that $\left.\phi\right\vert_H^{-1} \circ \phi$ defines a homomorphism from $\BK_c$ to $\BH_c$, a contradiction.

Let $\BJ$ be any substructure of $\BG$ which is a critical obstruction of $\BH$. It follows easily from the criticality of $\BK_c$ that $\BJ$ contains $\BK$. Since $\BH$ has finite duality it follows that there is a bound on the size of $\BJ$ and, hence, of $\BK$.
\end{proof}

\begin{remark}
{\normalfont
Notice that for the implications $(\ref{A1c})\Rightarrow(\ref{D1c})$ and $(\ref{D1c})\Rightarrow(\ref{D2c})$ we didn't need the fact that $\BH$ is a core.
}
\end{remark}

It has been shown in \cite{Nesetal08} that if a $\tau$-structure $\BH$ has finite duality, then there exists some finite set $\{\mathbb{T}_1,\dots,\mathbb{T}_m\}$ of $\tau$-trees such that for every $\tau$-structure $\BI$ not homomorphic to $\BH$, there exists some $1 \leq i \leq m$ such that $\mathbb{T}_i$ is homomorphic to $\BI$ but not homomorphic to $\BH$. We want to note that the equivalence between statements $(\ref{D2c})$ and $(\ref{D1c})$ does not follow from this fact. Indeed, direction $(\ref{D2c})\Rightarrow(\ref{D1c})$ does not hold when we do not require that the $\tau$-structure $\BH$ is equipped with constants as witnessed by the case when $\BH$ is the oriented $3$-cycle. Note that, in this case, $\BH$ satisfies $(\ref{D2c})$ since every $\tau$-tree is homomorphic to $\BH$ and, hence, $\BH$ has no critical $\tau$-tree obstructions at all. However, since any oriented cycle whose length is not a multiple of $3$ is a critical obstruction of $\BH$, it follows that $\BH$ does not satisfy $(\ref{D1c})$.

The next corollary is direct.

\begin{corollary}
\label{cor:fd}
Let $\BH$ be a finite $\tau$-structure which is a core. Then, the following are equivalent:
\begin{enumerate}
\item[(\mylabel{A1c}{A1c})] $\BH^2$ dismantles to its full diagonal;
\item[(\mylabel{B1c}{B1c})] $C(\BG,\BH)$ is $H$-connected for every locally finite $\tau$-structure $\BG$;
\item[(\mylabel{C1c}{C1c})] $\Hom(\BG,\BH)$ is topologically strong spatial mixing for every $\tau$-structure $\BG$; and
\item[(\mylabel{D1c}{D1c})] $\BH$ has finitely many critical obstructions.
\end{enumerate}
\end{corollary}

Notice that Corollary \ref{cor:fd} has novel consequences even in the graph case. Despite graphs that have finitely many critical obstructions are the graphs without any edge, one can always consider, given a graph $\BH$, a version of it endowed with all constants. Formally, we consider the $\tau$-structure $\overline{\BH}$ with $\tau = \{E\} \cup \{R_a\}_{a \in H}$, universe $H$, $E(\overline{\BH}) = E(\BH)$, and $R_a(\BH) = \{a\}$. 

First, notice that, since $\overline{\BH}$ contains all constants, $\overline{\BH}$ is a core. In addition, if $\BH$ dismantles to its full diagonal, then $\overline{\BH}$ also dismantles to its full diagonal. Therefore, by the previous results, $\overline{\BH}$ has a finite set $\{\mathbb{O}_1,\dots,\mathbb{O}_m\}$ of critical obstructions. Now suppose that we are given another graph $\BG$ and a partial $H$-coloring of $\BG$, this is to say, a function $\phi_U: U \to H$, where $U \subseteq G$. A natural question  is whether $\phi_U$ can be extended to a homomorphism $\phi \in \Hom(\BG,\BH)$. This computational problem is usually called  \emph{homomorphism extension problem} \cite{LaroseZ03}, and turns out to be equivalent to the \emph{retraction problem} \cite{FederH98}.

Consider the auxiliary $\tau$-structure $\overline{\BG}^{\phi_U}$ with universe $G$, $E(\overline{\BG}^{\phi_U}) = E(\BG)$, and $R_a(\overline{\BG}^{\phi_U}) = \{\phi_U^{-1}(a)\}$ for $a \in H$. The previous results tells us that such extension exists if and only if $\Hom(\mathbb{O}_i,\overline{\BG}^{\phi_U}) = \emptyset$ for all $1 \leq i \leq m$. In other words, we only need to check $\overline{\BG}^{\phi_U}$ locally to see that this is possible. It is easy to extend this observation to the more general framework of relational structures.

\section*{Acknowledgements}

We thank the anonymous reviewers for their careful reading of our manu\-script and their many insightful comments and suggestions.

The first author was supported by CONICYT/FONDECYT Postdoctorado 3190191 and ERC Starting Grants 678520 and 676970. The second author was supported by an NSERC Discovery grant. The third author was supported by MICCIN grants TIN2016-76573-C2-1P and PID2019-109137GB-C22,  and Maria de Maeztu Units of Excellence Programme MDM-2015-0502. The fourth author was supported by an NSERC Discovery grant and FRQNT.

This work was done in part while the first three authors were visiting the Simons Institute for the Theory of Computing at University of California, Berkeley.





\bibliographystyle{plain} 
\bibliography{references}



\end{document}